\documentclass[a4paper,11pt,oneside]{article}

\usepackage{amsmath,amsthm,amsfonts,amssymb,amscd}
\usepackage{mathtools}
\usepackage{float}
\usepackage{subfig} 
\usepackage{enumerate}
\usepackage{booktabs}
\usepackage[noend]{algpseudocode}
\usepackage[normalem]{ulem}
\usepackage{graphicx,bm,xcolor}
\usepackage{algorithm}
\usepackage{algpseudocode}
\usepackage[numbers]{natbib}
\usepackage[hidelinks]{hyperref}
\usepackage{caption}
\usepackage[T1]{fontenc}
\usepackage{t1enc}
\usepackage[polish,german,english]{babel}
\usepackage{verbatim}
\usepackage{diagbox}
\usepackage{colortbl}
\usepackage{siunitx}
\usepackage{mathdots} 
\usepackage{bbm}

\usepackage{url}
\usepackage{authblk}

\algnewcommand{\algorithmicgoto}{\textbf{go to}}
\algnewcommand{\Goto}[1]{\algorithmicgoto~\ref{#1}}

\newcommand{\R}{\mathbb{R}}

\definecolor{orange}{RGB}{230, 159, 0}
\definecolor{skyblue}{RGB}{86, 180, 233}
\definecolor{yellow}{RGB}{240, 228, 66}
\definecolor{blue}{RGB}{0, 114, 178}
\definecolor{vermillion}{RGB}{213, 94, 0}

\usepackage[top=2cm, bottom=2cm, left=2cm, right=2cm]{geometry}

\theoremstyle{plain} 
\newtheorem{theorem}{Theorem}[section]

\newtheorem{remark}[theorem]{Remark}

\newtheorem{form}[theorem]{Formulation}

\theoremstyle{definition} %

\theoremstyle{remark} %

\DeclareMathOperator*{\argmax}{arg\,max}

\DeclareMathOperator*{\dG}{dG}
\DeclareMathOperator*{\cG}{cG}
\DeclareMathOperator*{\Ieff}{I_{\text{eff}}}

\newcommand\restrict[1]{\raisebox{-.5ex}{$|$}_{#1}}

\definecolor{brewer1}{HTML}{A6CEE3}
\definecolor{brewer2}{HTML}{1F78B4}
\definecolor{brewer3}{HTML}{B2DF8A}
\definecolor{brewer4}{HTML}{33A02C}

\begin{document}


\title{MORe DWR: Space-time goal-oriented error control \\ for incremental POD-based ROM}
\author[1,2]{Hendrik Fischer}
\author[1,2]{Julian Roth}
\author[1,2]{Thomas Wick}
\author[2]{Ludovic Chamoin}
\author[2]{Amélie Fau}

\affil[1]{Leibniz Universit\"at Hannover, Institut f\"ur Angewandte
  Mathematik, \linebreak AG Wissenschaftliches Rechnen, Welfengarten 1, 30167 Hannover, Germany}
\affil[2]{Universit\'e Paris-Saclay, CentraleSupélec, ENS Paris-Saclay, CNRS, LMPS - Laboratoire de Mécanique Paris-Saclay,
91190 Gif-sur-Yvette, France}

\date{}
\maketitle

\setcounter{page}{1}

\begin{abstract}
In this work, the dual-weighted residual (DWR) method is applied to obtain a certified incremental proper orthogonal decomposition (POD) based reduced order model. A novel approach called MORe DWR (\underline{M}odel \underline{O}rder \underline{Re}duction with \underline{D}ual-\underline{W}eighted \underline{R}esidual error estimates) is being introduced. It marries tensor-product space-time reduced-order modeling with time slabbing and an incremental POD basis generation with 
goal-oriented error control based on dual-weighted residual estimates. The error in the goal functional is being estimated during the simulation and the POD basis is being updated if the estimate exceeds a given threshold. 
This allows an adaptive enrichment of the POD basis in case of unforeseen changes in the solution behavior which is of high interest in many real-world applications.
Consequently, the offline phase can be skipped, the reduced-order model is being solved directly with the POD basis extracted from the solution on the first time slab and \mbox{--if necessary--} the POD basis is being enriched on-the-fly during the simulation with high-fidelity finite element solutions. Therefore, the full-order model solves can be reduced to a minimum, which is demonstrated on numerical tests for the heat equation and elastodynamics. 
\end{abstract}

\section{Introduction}
\label{sec:intro}
Model order reduction (MOR) by means of the proper orthogonal decomposition (POD) has been applied for cheap surrogate modeling to a plethora of partial differential equations (PDEs) \cite{sirovich1987turbulence, kunisch_volkwein_2001, Luo2009,ballarin_supremizer_2015,volkwein_proper_2013,benner2020model,Kerschen2005,BeCoOhlWill15,ABBASZADEH2021109875,WANG2020109402,GIRFOGLIO2022105536}.
Therein, the dynamics is projected onto a set of POD modes that constitute an approximate basis for the solution manifold to reduce the cost of running expensive high-fidelity simulations.
This proper orthogonal decomposition based reduced-order modeling (POD-ROM) is a truth approximation because it yields a compressed representation of an a priori known solution trajectory.
To avoid the necessity of these expensive high-fidelity simulations beforehand, we use error estimates to only locally perform high-fidelity calculations.

The dual-weighted residual method is used in this work to switch between ROM and high-fidelity computations.
The space-time dual-weighted residual (DWR) method is an extension of the DWR method for stationary problems introduced in \cite{BeRa96, becker_rannacher_2001,bangerth_rannacher_2003}, which is based on seminal prior work of Johnson and co-workers \cite{ErikssonEstepHansboJohnson1995}. 
A recent overview on the usage with adaptive predictive multiscale modeling was published by Oden \cite{Od18}.
The space-time DWR method has been applied to parabolic PDEs by Schmich (Besier) and Vexler \cite{SchmichVexler2008}, Schmich (Besier) \cite{Schmich2009} and Besier and Rannacher \cite{BeRa12} and in the authors' own works \cite{ThiWi22_arxiv, RoThiKoeWi2022}. 
Moreover, it has been applied to hyperbolic PDEs in the dissertation of Rademacher \cite{Rade09} and to the wave equation by Bangerth et al. \cite{BangerthGeigerRannacher2010}.
Since the theory for the error estimation is formulated in spatio-temporal function spaces and requires space-time variational formulations, we employ
a space-time finite element method (FEM) discretization; see for instance \cite{LaStein19}.
Space-time finite elements for the heat equation have been studied in \cite{SchmichVexler2008, Schaf21} and for the elastodynamics equation in \cite{HuHu88,BangerthGeigerRannacher2010}.
Similar space-time FEM implementations can be found in FEniCS in \cite{Loveland2022} and in NGSolve in \cite{Lehrenfeld2021, Preuss2018}. 

In recent years, space-time formulations have been applied to model order reduction \cite{Choi2019, kim2021efficient, Tenderini2022, Ekre2020}, including a windowed space-time approach for more efficiency \cite{Shimizu2021}. 
Additional applications of space-time model order reduction include optimal control \cite{Zoccolan2023} and classical error estimates and hyper-reduction estimates using discrete empirical interpolation \cite{Bernreuther2022}.
A lot of research on DWR error estimates for hyper-reduction with reduced quadrature rules has been done by Yano \cite{Yano2020, Sleeman2022}. Another reduced-order modeling approach employing goal-oriented error estimates has been proposed by Meyer and Matthies \cite{Meyer2003}, where the estimates have been used to remove POD basis vectors that are not relevant for the accurate computation of the quantity of interest. 
Finally, related methods include the proper generalized decomposition (PGD) \cite{chinesta2011} and hierarchical model (HiMod) reduction \cite{Perotto2015_spacetime, Perotto2016_fluid, Perotto2020_pgd}, which uses estimates for the POD in the transverse direction of the dynamics.

In this work, we propose a different methodology for POD-ROM computations in which only a small portion of the solution trajectory is being computed with the expensive full-order-model (FOM) and the reduced-order-model (ROM) is being updated on-the-fly when the error estimates exceed a prescribed tolerance. 
This is being accomplished by combining POD-ROM with the incremental POD and space-time dual-weighted residual error estimates. 
We work out the algorithmic details resulting in a final newly 
proposed algorithm for incremental ROM.
The incremental POD method relies on additive rank-b updates of the singular value decomposition \cite{brand2002incremental, brand2006fast} and has successfully been applied to the incremental model order reduction of fluid flows \cite{kuhl2023incremental}. 
As an overall framework, we employ a space-time setting.
More concretely, we rely on the tensor-product space-time FEM implementation from \cite{RoThiKoeWi2022} based on the FEM library deal.II \cite{dealii2019design, dealII94}. The final algorithm is implemented and 
demonstrated with various settings that include parabolic problems (heat equation) and second-order hyperbolic problems (elastodynamics). The main objective is to show the decrease in computational cost by keeping the accuracy of the numerical solutions. Moreover, the error estimator and the goal functional are compared in terms of effectivities. 

The outline of this paper is as follows: In Section \ref{sec:problem}, we formulate the problem for the heat equation and elastodynamics and discretize them with tensor-product space-time finite elements. Next, in Section \ref{sec:ROM} we recapitulate POD-based reduced-order modeling and depict its extension to tensor-product space-time POD-ROM. Then, in Section \ref{sec:error_estimator_certified_ROM} the theories for the space-time error estimates and the incremental model order reduction are elucidated. In Section \ref{sec:numerical_tests}, numerical tests in 1+1D, 2+1D and 3+1D are being conducted for the heat equation and elastodynamics. Finally, our findings are summarized in Section \ref{sec:conclusion}.

\section{Problem formulation and discretization}
\label{sec:problem}

\subsection{Model problem formulation}
Let $\tilde{d}\in\mathbb{N}$ with ${\tilde{d}}$ depending on whether the problem is vector- or scalar-valued, i.e. for the heat equation we have ${\tilde{d}} = 1$, whereas for elastodynamics in $u$-formulation (where $u$ denotes the displacements) we have ${\tilde{d}} = d$ and for the $(u,v)$-formulation (where $u$ is as before and $v$ denotes the velocity), 
we have ${\tilde{d}} = 2d$, where 
$d \in \{1,2,3\}$ is the spatial dimension.
In the problem description, $I := (0,T)$ denotes the temporal 
domain and $\Omega \subset \mathbb{R}^d$
a sufficiently smooth spatial domain. Here, the spatial boundary is split into a Dirichlet boundary $\Gamma_D \subseteq \partial\Omega$ and a Neumann boundary $\Gamma_N \subsetneq \partial\Omega$ with $\Gamma_D \cap \Gamma_N = \emptyset$. 
We consider the abstract time-dependent problem:
Find $u: \bar{\Omega} \times \bar{I} \rightarrow \mathbb{R}^{\tilde{d}}$ such that
\begin{equation}
\begin{aligned}\label{eq:time_dependent_problem}
    \partial_t u + \mathcal{A}(u) &= f \qquad\quad \text{in } \Omega \times I , \\
    u &= u_D \qquad \text{on } \Gamma_D \times I, \\
    \mathcal{B}(u) &= g_N \qquad \text{on } \Gamma_N \times I, \\
    u &= u^0 \,\,\qquad \text{in } \Omega \times \{ 0 \},
\end{aligned}
\end{equation}
with possibly nonlinear spatial operator $\mathcal{A}$, boundary operator $\mathcal{B}$ and sufficiently regular right-hand side $f$. 
Choosing a suitable continuous spatial function space $V := V(\Omega)$, a continuous temporal functional space $X := X(I, \cdot)$ and time-dependent Sobolev space $X(I, V(\Omega))$ mapping from $I$ into $V(\Omega)$, we can define the continuous spatio-temporal variational formulation as: Find $u \in u_D + X(I, V(\Omega))$ such that 
\begin{align*}
    A(u)(\varphi) &:= (\!(\partial_t u,\varphi)\!)  + (\!(\mathcal{A}(u),\varphi)\!) + (u(0),\varphi(0)) \\
    &= (\!(f,\varphi)\!) + \langle\!\langle g_N - \mathcal{B}(u), \varphi \rangle\!\rangle_{\Gamma_N} + (u^0, \varphi(0)) =: F(\varphi) \qquad \forall \varphi \in X(I, V(\Omega)),
\end{align*}
where we use the notation
\begin{align*}
    (f,g) := (f,g)_{L^2(\Omega)} := \int_\Omega f \cdot g\ \mathrm{d}x, \qquad (\!(f,g)\!) := (f,g)_{L^2(I, L^2(\Omega))} := \int_I (f, g)\ \mathrm{d}t, \\
    \langle f,g \rangle := \langle f,g \rangle_{L^2(\Gamma)} := \int_\Gamma f \cdot g\ \mathrm{d}s, \qquad \langle\!\langle f,g\rangle\!\rangle := (f,g)_{L^2(I, L^2(\Gamma))} := \int_I \langle f, g\rangle\ \mathrm{d}t.
\end{align*}
In this notation, $f \cdot g$ represents the Euclidean inner product if $f$ and $g$ are scalar- or vector-valued and it stands for the Frobenius inner product if $f$ and $g$ are matrices.
We notice that some partial differential equations (PDE) which fall into this framework are the heat equation and more generally parabolic problems. With a bit of abuse of notation,
elastodynamics formulated as a first-order-in-time system can also be written in the above form, which we however precise below for the sake of mathematical precision.

\subsubsection{Heat equation}
The strong formulation of the heat equation reads: Find the temperature $u: \bar{\Omega} \times \bar{I} \rightarrow \mathbb R$ such that
\begin{align*}
    \partial_{t} u - \Delta_x u = f \qquad \quad \text{in } \Omega \times I,
\end{align*}
with $\mathcal{A}(u) := - \Delta_x u$ in (\ref{eq:time_dependent_problem}).
The initial and boundary conditions are given by
\begin{align*}
    u &= u^0 \qquad \text{on } \Omega \times \{ 0 \}, \\
    u &= 0 \qquad \text{on } \partial\Omega \times I.
\end{align*}
We thus arrive at the continuous variational formulation:
\begin{form}[Continuous variational formulation of the heat equation]\label{form:variational_heat}\ \\
    Find $u \in X(I, V(\Omega)) := \{v \in L^2(I, H^1_0(\Omega)) \mid  \partial_t v \in L^2(I, (H^1_0(\Omega))^\ast) \}$ such that
    \begin{align*}
        A(u)(\varphi) := (\!(\partial_t u,\varphi)\!)  + (\!(\nabla_x u, \nabla_x \varphi)\!) + (u(0),\varphi(0)) = (\!(f,\varphi)\!) + (u^0, \varphi(0)) =: F(\varphi) \qquad \forall \varphi \in X(I, V(\Omega)).
    \end{align*}
\end{form}
\noindent For this variational formulation, we use $u_0 \in L^2(\Omega)$ and $f \in L^2(I, H^1_0(\Omega)^\ast)$ \cite{wloka1987}. Here $H^1_0(\Omega)^\ast$ denotes the dual space of $H^1_0(\Omega)$.

\subsubsection{Elastodynamics equation}
\label{sec:elasto_equation}
The strong formulation of linear elastodynamics in three spatial dimensions reads:
Find the displacement $u: \bar{\Omega} \times \bar{I} \rightarrow \mathbb{R}^d$ such that 
\begin{align*}
    \partial_{tt} u &- \nabla_x \cdot \sigma(u) = 0 \qquad \quad \text{in } \Omega \times I,
\end{align*}
with
\begin{align*}
   \sigma(u) &= 2 \mu E(u) + \lambda \operatorname{tr}(E(u))\mathbbm{1}_{d \times d}, \tag{{stress tensor}} \\
   E(u) &= \frac{1}{2}(\nabla_x u + (\nabla_x u)^T), \tag{{linearized strain tensor}}
\end{align*}
where $\mathbbm{1}_{d \times d} \in \mathbb{R}^{d\times d}$ is the identity matrix and the Lamé parameters are $\mu > 0$ and $\lambda > -\frac{2}{3}\mu$.
The initial conditions are given by 
\begin{align*}
    u &= u^0 \qquad \text{on } \Omega \times \{ 0 \}, \\
    \partial_t u &= v^0 \qquad \text{on } \Omega \times \{ 0 \}.
\end{align*}
As boundary conditions, we prescribe
\begin{align*}
    u &= 0 \qquad\quad \text{on } \Gamma_D \times I, \\ 
    \mathcal{B}(u) = \sigma(u) \cdot n &= g_N \qquad\,\, \text{on } \Gamma_N \times I.
\end{align*}
We convert this into a first-order system in time and solve for displacement $u: \bar{\Omega} \times \bar{I} \rightarrow \mathbb{R}^d$ and velocity $v: \bar{\Omega} \times \bar{I} \rightarrow \mathbb{R}^d$ such that
\begin{align*}
    \partial_t v - \nabla_x \cdot \sigma(u) &= f \qquad \quad \text{in } \Omega \times I, \\
    \partial_t u - v &= 0 \qquad \quad \text{in } \Omega \times I,
\end{align*}
with $\mathcal{A}(u,v) := -\nabla_x \cdot \sigma(u) - v$ in (\ref{eq:time_dependent_problem}).
We still have the same initial and boundary conditions with the only difference that we now have
\begin{align*}
    v &= v^0 \qquad \text{on } \Omega \times \{ 0 \}, \\
    v &= 0 \qquad\quad \text{on } \Gamma_D \times I.
\end{align*}
For the variational formulation, we use $u_0 \in H^1_{\Gamma_D, 0}(\Omega)^d$, which is the space of weakly differentiable functions that vanish on $\Gamma_D$, $ v_0 \in L^2(\Omega)^d, g_N \in L^2(I, L^2(\Gamma_N)^d)$ and the function spaces
\begin{align*}
    X(I, V^u(\Omega)) &:= \{ v \in L^2(I,H^1_{\Gamma_D, 0}(\Omega)^d) \mid  \partial_t v \in L^2(I, L^2(\Omega)^d), \partial^2_t v \in L^2\left(I, (H^1_{\Gamma_D, 0}(\Omega)^d)^\ast\right)  \}, \\
    X(I, V^v(\Omega)) &:= \{ v \in L^2(I, L^2(\Omega)^d) \mid  \partial_t v \in L^2\left(I, (H^1_{\Gamma_D, 0}(\Omega)^d)^\ast\right)  \},\\
    X(I, V(\Omega)) &:=  X(I, V^u(\Omega)) \times  X(I, V^v(\Omega)).
\end{align*}
\noindent We thus solve the continuous variational formulation: 
\begin{form}[Continuous variational formulation of the elastodynamics equation]\label{form:variational_elasto}\ \\
    Find $U = (u, v) \in X(I, V(\Omega))$ such that
    \begin{align*}
        A(U)(\Phi) = F(\Phi) \qquad \forall \Phi = (\varphi^u, \varphi^v) \in X(I, V(\Omega)),
    \end{align*}
    where
    \begin{align*}
        A(U)(\Phi) &:= (\!( \partial_t v, \varphi^u)\!) + (\!( \sigma(u), \nabla_x \varphi^u)\!) + (v(0), \varphi^u(0)) + (\!( \partial_t u, \varphi^v)\!) - (\!( v, \varphi^v)\!) + (u(0), \varphi^v(0)), \\
         F(\Phi) &:=(v^0, \varphi^u(0)) + \langle\!\langle g_N, \varphi^u \rangle\!\rangle_{\Gamma_N} + (u^0, \varphi^v(0)).
    \end{align*}
\end{form}

\subsection{Tensor-product space-time FEM discretization}\label{subsec:tensorproduct_spacetime_FEM_discretization}
We follow our recent work on space-time adaptivity for the Navier-Stokes equations \cite{RoThiKoeWi2022} and use tensor-product space-time finite elements (FEM) with discontinuous finite elements in time ($\dG$) and continuous finite elements in space ($\cG$). Using the tensor-product of the temporal and spatial basis functions is a special case of the broad class of space-time finite element methods \cite{LaStein19}. We will now explain tensor-product space-time FEM at the example of the heat equation, where the function spaces can be found in \cite{SchmichVexler2008} and the slabwise tensor-product space-time implementation is being outlined in \cite{ThiWi22_arxiv}. We assume that the spatial mesh remains fixed, which simplifies the analysis and the implementation. Furthermore, we outline the extension of this methodology to elastodynamics.

\subsubsection{Discretization in time}
Let $\mathcal{T}_k := \{ I_m := (t_{m-1}, t_m) \mid 1 \leq m \leq M \}$ be a partitioning of time, i.e. $ \bar{I} = [0,T] = \bigcup_{m = 1}^M \bar{I}_m $. 
We now introduce broken continuous level function spaces
\begin{align*}
    \tilde{X}(\mathcal{T}_k, V(\Omega)) := \{ v \in L^2(I, L^2(\Omega)) \mid v\restrict{I_m} \in X(I_m, V(\Omega))\quad \forall I_m \in \mathcal{T}_k \}
\end{align*}
for the heat equation and 
\begin{align*}
    \tilde{X}(\mathcal{T}_k, V^u(\Omega)) &:= \{ v \in L^2(I, L^2(\Omega)^3) \mid v\restrict{I_m} \in X(I_m, V^u(\Omega))\quad \forall I_m \in \mathcal{T}_k \}, \\
    \tilde{X}(\mathcal{T}_k, V^v(\Omega)) &:= \{ v \in L^2(I, L^2(\Omega)^3) \mid v\restrict{I_m} \in X(I_m, V^v(\Omega))\quad \forall I_m \in \mathcal{T}_k \}, \\
    \tilde{X}(\mathcal{T}_k, V(\Omega)) &:= \tilde{X}(\mathcal{T}_k, V^u(\Omega)) \times \tilde{X}(\mathcal{T}_k, V^v(\Omega))
\end{align*}
for the elastodynamics equation. These broken function spaces \cite{DiPietroErn2012} are required, since we want to perform a conforming discontinuous Galerkin discretization in time and thus need to allow for discontinuities between time intervals/temporal elements. 
Due to these discontinuities, we define the limits of $f$ at time $t_m$ from above and from below for a function $f$ as
\begin{align*}
    f_{m}^\pm := \lim_{\epsilon \searrow 0} f(t_m \pm \epsilon),
\end{align*}
and the jump of the function value of $f$ at time $t_m$ as
\begin{align*}
    [f]_m := f_{m}^+ - f_{m}^-.
\end{align*}
The function spaces enable us to include discontinuities in the variational formulations:
\begin{form}[Time-discontinuous variational formulation of the heat equation]\label{form:heat_with_jumps}\ \\
    Find $u \in \tilde{X}(\mathcal{T}_k, V(\Omega))$ such that
    \begin{align*}
        \tilde{A}(u)(\varphi) = \tilde{F}(\varphi) \qquad \forall \varphi \in \tilde{X}(\mathcal{T}_k, V(\Omega)),
    \end{align*}
    where
    \begin{align*}
        \tilde{A}(u)(\varphi) &:= \sum_{m = 1}^M \int_{I_m} (\partial_t u,\varphi)  + (\nabla_x u,\nabla_x \varphi)\ \mathrm{d}t + \sum_{m = 1}^{M-1}([u]_m, \varphi_{m}^{+}) + (u_0^+,\varphi_0^+), \\
        \tilde{F}(\varphi) &:= (\!(f,\varphi)\!) + (u^0, \varphi_0^+).
    \end{align*}
\end{form}
\begin{form}[Time-discontinuous variational formulation of the elastodynamics equation]\label{form:elasto_with_jumps}\ \\
    Find $U = (u, v) \in \tilde{X}(\mathcal{T}_k, V(\Omega))$ such that
    \begin{align*}
        \tilde{A}(U)(\Phi) = \tilde{F}(\Phi) \qquad \forall \Phi = (\varphi^u, \varphi^v) \in \tilde{X}(\mathcal{T}_k, V(\Omega)),
    \end{align*}
    where
    \begin{align*}
        \tilde{A}(U)(\Phi) &:= \sum_{m = 1}^M \int_{I_m} (\partial_t v, \varphi^u) + ( \sigma(u), \nabla_x \varphi^u) + (\partial_t u, \varphi^v) - ( v, \varphi^v)\ \mathrm{d}t \\
        &\qquad +\sum_{m = 1}^{M-1}([v]_m, \varphi_m^{u,+}) + ([u]_m, \varphi_m^{v,+}) + (v_0^+, \varphi_0^{u,+})  + (u_0^+, \varphi_0^{v,+}), \\
        \tilde{F}(\Phi) &:= (v^0, \varphi_0^{u,+}) + \langle\!\langle g_N, \varphi^u \rangle\!\rangle_{\Gamma_N} + (u^0, \varphi_0^{v,+}).
    \end{align*}
\end{form}
\noindent We have the inclusions $X(I, \cdot) \subset \tilde{X}(\mathcal{T}_k, \cdot)$, since for continuous functions the jump terms vanish, and thus the variational Formulation \ref{form:heat_with_jumps} and Formulation \ref{form:elasto_with_jumps} are consistent. 

Next, we define the semi-discrete space for the heat equation as
\begin{align*}
    X_k^{\dG(r)}(\mathcal{T}_k, V(\Omega)) := \left\lbrace v_k \in L^2(I, L^2(\Omega))\, \middle| \, v_k\restrict{I_m} \in P_r(I_m, H^1_{0}(\Omega)) \right\rbrace \subset \tilde{X}(\mathcal{T}_k, V(\Omega))
\end{align*}
and for the elastodynamics equation as
\begin{align*}
    X_k^{\dG(r)}(\mathcal{T}_k, V^u(\Omega)) &:= \left\lbrace v_k \in L^2(I, L^2(\Omega)^3)\, \middle| \, v_k\restrict{I_m} \in P_r(I_m, H^1_{\Gamma_D, 0}(\Omega)^3) \right\rbrace \subset \tilde{X}(\mathcal{T}_k, V^u(\Omega)), \\
    X_k^{\dG(r)}(\mathcal{T}_k, V^v(\Omega)) &:= X_k^{\dG(r)}(\mathcal{T}_k, V^u(\Omega)), \\
    X_k^{\dG(r)}(\mathcal{T}_k, V(\Omega)) &:= X_k^{\dG(r)}(\mathcal{T}_k, V^u(\Omega)) \times X_k^{\dG(r)}(\mathcal{T}_k, V^v(\Omega)),
\end{align*}
where the space-time function spaces $\tilde{X}(\mathcal{T}_k, \cdot)$ have been discretized in time with the discontinuous Galerkin method of order $r\in \mathbb{N}_0$ ($\dG(r)$). Typical choices in our work for the temporal degree are $r = 1$ and $r = 2$. Here, $P_r(I_m, Y)$ is the space of polynomials of order $r$, which map from the time interval $I_m$ into the space $Y$.
The $\dG(r)$ time discretization for the case $r = 1$ is illustrated in Figure~\ref{fig:dG1_time_discretization}.
\begin{figure}[H]
    \centering
    \includegraphics[scale=1.2]{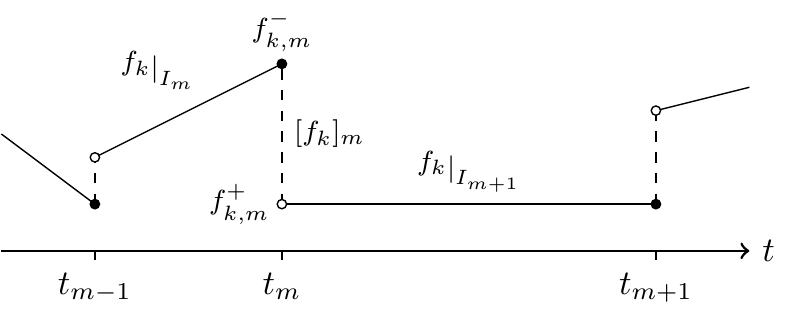}
    \caption{dG(1) time discretization}
    \label{fig:dG1_time_discretization}
\end{figure}
\noindent The locations of the temporal degrees of freedom (DoFs) are defined by quadrature rules. Due to the discontinuity of the temporal discretization, various quadrature rules can be chosen, the most common being Gauss-Lobatto, Gauss-Legendre and Gauss-Radau. In Figure \ref{fig:dG1_time_discretization} the location of the temporal degrees of freedom are chosen at the ends of the time intervals, which corresponds to Gauss-Lobatto quadrature. In Section~\ref{sec:numerical_tests}, we use Gauss-Legendre and Gauss-Lobatto quadrature in time to demonstrate the versatility of our method concerning the choice of the temporal quadrature formula.

It has been derived in \cite{DeHaTr81} (see also the classical textbooks \cite{Rannacher2017, eriksson1996}) that the $\dG(0)$ time-discretization is a variant of the backward Euler scheme. Higher-order schemes 
are derived as well and it was established that $dG(r_p)$ discretizations, where $r_p\in\mathbb{N}_0$ is the polynomial degree, are generically implicit and $A$-stable.

\subsubsection{Discretization in space}
For the spatial discretization of the variational formulation, we use a fixed mesh $\mathcal{T}_h$, which consists of intervals in one dimension and of quadrilateral (2D) or hexahedral (3D) elements in higher dimensions. We can then use element-wise polynomial functions of up to order $s \in \mathbb{N}$ as our spatial function space, i.e.,
\begin{align*}
    V_h^s := V_h^s(\mathcal{T}_h) := \left\lbrace v \in  C(\bar \Omega) \middle| v\restrict{K} \in \mathcal{Q}_{s}(K) \quad  \forall K \in \mathcal{T}_h \right\rbrace
\end{align*}
and for the elastodynamics equation
\begin{align*}
    V_h^{s,u} := V_h^{s,u}(\mathcal{T}_h) := \left\lbrace v \in  C(\bar \Omega)^d \middle| v\restrict{K} \in (\mathcal{Q}_{s}(K))^d \quad  \forall K \in \mathcal{T}_h \right\rbrace =: V_h^{s,v}(\mathcal{T}_h) =: V_h^{s,v},
\end{align*}
where $\mathcal{Q}_{s}(K)$ is being constructed by mapping tensor-product polynomials of degree $s$ from the master element $(0,1)^d$ to the element $K$. The fully discrete function space for the heat equation is then given by
\begin{align*}
    X_k^{\dG(r)}(\mathcal{T}_k, V_h^s) := \left\lbrace v_{kh} \in L^2(I, L^2(\Omega))\, \middle| \, v_{kh}\restrict{I_m} \in P_r(I_m, V_h^s) \quad \forall I_m \in \mathcal{T}_k \right\rbrace
\end{align*}
and for the elastodynamics equation
\begin{align*}
    X_k^{\dG(r)}(\mathcal{T}_k, V_h^{s}) &:= \left\lbrace v_{kh} \in L^2(I, L^2(\Omega)^{2d})\, \middle| \, v_{kh}\restrict{I_m} \in P_r(I_m, V_h^{s}) \quad \forall I_m \in \mathcal{T}_k \right\rbrace, \\
    V_h^{s} &:= V_h^{s,u} \times V_h^{s, v}.
\end{align*}
Thus, the fully discrete variational formulation reads for the heat equation:\\
\indent Find $u_{kh} \in X_k^{\dG(r)}(\mathcal{T}_k, V_h^s)$ such that 
\[
\tilde{A}(u_{kh})(\varphi_{kh}) = \tilde{F}(\varphi_{kh}) \quad \forall \varphi_{kh} \in X_k^{\dG(r)}(\mathcal{T}_k, V_h^s).
\]
Moreover, the fully discrete variational formulation for the elastodynamics equation reads:\\
\indent Find $U_{kh} := (u_{kh}, v_{kh}) \in X_k^{\dG(r)}(\mathcal{T}_k, V_h^{s})$ such that
\begin{align*}
    \tilde{A}(U_{kh})(\Phi_{kh}) = \tilde{F}(\Phi_{kh}) \quad \forall \Phi_{kh} = (\varphi_{kh}^u, \varphi_{kh}^v) \in X_k^{\dG(r)}(\mathcal{T}_k, V_h^{s}).
\end{align*}

\subsubsection{Slabwise discretization} \label{subsection:slabwise_discretization}
Finally, we want to remark that the fully discrete variational formulations do not need to be solved on the entire space-time cylinder $\Omega \times I$, but can also be solved sequentially on space-time slabs
\begin{align*}
   S_l^n :=  \Omega \times \left( \bigcup\limits_{m = l}^{n} I_m \right),
\end{align*}
where $1 \leq l \leq n \leq M$, see also \cite{ThiWi22_arxiv}[Remark 2.1]. As mentioned previously, we can then get the space-time FEM basis on $S_l^n$ by taking the tensor-product of the spatial and the temporal finite element basis functions. This simplifies the finite element discretization of the abstract time-dependent problem (\ref{eq:time_dependent_problem}), since the main prerequisite is a FEM code for the stationary problem $\mathcal{A}(u) = f  \text{ in } \Omega$. Furthermore, tensor-product space-time FEM allows for larger flexibility in the choice of temporal discretization, since changing the temporal degree of the space-time discretization can be performed simply by
changing the polynomial degree of the temporal finite elements. Due to the tensor-product structure of the space-time FE basis, it is straightforward how proper orthogonal decomposition (POD) based reduced-order modeling can be performed, since on an abstract level only the spatial finite element basis needs to be replaced by the spatial POD basis. 

For the heat equation on the space-time slab $S_l^n$ with $n-l+1$ time intervals, we arrive at the linear equation system 
\begin{align}
\label{eq:LES_FOM}
    \begin{pmatrix}
        A & & &  & \bm{0} \\
        B & A & & &  \\
         & B & A &  &  \\
         & & \ddots & \ddots & \\
        \bm{0} & & & B & A
    \end{pmatrix}
    \begin{pmatrix}
        U_{l} \\
        U_{l + 1} \\
        U_{l + 2} \\
        \vdots \\
        U_{n}
    \end{pmatrix}
    = 
    \begin{pmatrix}
        F_{l} - B U_{l-1} \\
        F_{l + 1} \\
        F_{l + 2} \\
        \vdots \\
        F_{n}
    \end{pmatrix} 
\end{align}
or in brevity 
\begin{align}\label{eq:LES_heat_fom}
    A_{S_l^n} U_{S_l^n} = F_{S_l^n}
\end{align}
with 
\begin{align*}
    A &= C_k \otimes M_h + M_k \otimes K_h,\\
    B &= -D_k \otimes M_h ,
\end{align*}
where we use the spatial matrices
\begin{align*}
    M_h &= \left\{ (\varphi_h^{(j)},\varphi_h^{(i)}) \right\}_{i,j = 1}^{\# \text{DoFs}(\mathcal{T}_h)}, \\
    K_h &= \left\{ (\nabla_x\varphi_h^{(j)},\nabla_x\varphi_h^{(i)}) \right\}_{i,j = 1}^{\# \text{DoFs}(\mathcal{T}_h)}
\end{align*}
and the temporal matrices
\begin{align*}
    M_k &= \left\{ \int_{I_m}\varphi_k^{(j)}\cdot \varphi_k^{(i)}\ \mathrm{d}t \right\}_{i,j = 1}^{\# \text{DoFs}(I_m)}, \\
    C_k &= \left\{ \int_{I_m}\partial_t\varphi_k^{(j)}\cdot \varphi_k^{(i)}\ \mathrm{d}t + \varphi_{k,m-1}^{(j),+}\cdot \varphi_{k,m-1}^{(i),+} \right\}_{i,j = 1}^{\# \text{DoFs}(I_m)}, \\
    D_k &= \left\{ \varphi_{k,m-1}^{(j),-}\cdot \varphi_{k,m-1}^{(i),+} \right\}_{i,j = 1}^{\# \text{DoFs}(I_m)}.
\end{align*}
Note that $U_{l}, \dots, U_{n}$ are space-time vectors themselves, where $U_m \in \mathbb{R}^{\# \text{DoFs}(I_m)\, \cdot\, \# \text{DoFs}(\mathcal{T}_h)}$ with $m = l, \dots, n$ is the coefficient vector of the solution $u_{kh}$ on the time interval $I_m$, i.e., for the $\dG(r)$ method in time with temporal quadrature points $t_1, \dots, t_{r+1}$ we have
\begin{align*}
    U_{m} = \begin{pmatrix}
        U_{m} (t_1) \\
        \vdots \\
        U_{m} (t_{r+1})
    \end{pmatrix}, \quad m = 1, \dots, M,
\end{align*}
where $M$ is the total number of time intervals.
In particular, if we use space-time slabs that contain only one temporal element, then we only need to solve the linear system
\begin{align*}
    AU_m = F_m - BU_{m-1}
\end{align*}
for each time slab $S_m := S_m^m = \mathcal{T}_h \times I_m$. For efficiency reasons, in the remainder of this paper, we only consider such slabs of size one.

For the elastodynamics equation, the space-time FEM linear system can be derived similarly. The linear system and time-stepping formulations for $\dG(1)$ and $\dG(2)$ with Gauss-Lobatto quadrature in time can be found in \ref{sec:elasto_time_stepping}. 

\begin{remark}
    Although the linear systems for the heat equation in this section and for the elastodynamics equation in \ref{sec:elasto_time_stepping} have been presented as the tensor product of spatial matrices, tensor-product space-time FEM can be applied to a much larger class of problems. For instance, it is not always possible to decompose a space-time linear system into this tensor-product structure when the PDE contains coefficients that depend on space and time. 
    Nevertheless, our implementation of tensor-product space-time FEM is general enough to also deal with these kinds of problems, since it does not rely on a tensor-product of the linear system but only on the tensor-product structure of the finite element basis.
\end{remark}

\section{Reduced-order modeling}
\label{sec:ROM}
\subsection{POD-ROM}
The increase in computational power in the last decades has made it possible to exploit high-performance computing for large-scale numerical simulations. Nevertheless, in some scenarios, e.g. for multiphysics problems, high-performance computing can be computationally expensive, in particular also having a large carbon footprint and enormous energy consumption. 
These circumstances motivate the application of model order reduction (MOR) techniques on the premise of a large computational speedup to satisfy these demands. In this work, we mainly deal with projection-based reduced basis methods (RBM) \cite{haasdonk2008reduced,hesthaven2016certified,bertagna2014model,lassila_model_2014,haasdonk2017reduced,rozza2008reduced,nguyen2009reduced} since this methodology aims at efficient treatments by providing both an approximate solution procedure and efficient error estimates \cite{haasdonk2008reduced}. Here, the critical observation is that instead of using projection spaces with general approximation properties (e.g. finite element method) problem-specific approximation spaces are chosen and then can be used for the discretization of the original problem \cite{rozza2005shape}. Based on these spaces and the assumption that the solution evolves smoothly in a low-dimensional solution manifold (equivalent to a small Kolmogorov N-width \cite{kolmogoroff1936uber,benner2020model,hesthaven2016certified}), a reduced-order model (ROM) can be constructed that represents with sufficient accuracy the physical problem of interest using a significantly smaller number of degrees of freedom \cite{rozza2005shape}. 

In order to construct the reduced spaces, the solution manifold is empirically explored  by means of solutions of the full-order model as developed in Section \ref{subsec:tensorproduct_spacetime_FEM_discretization}. Then, a proper orthogonal decomposition (POD) is conducted on these snapshots of the high-fidelity solution
to obtain the reduced basis functions
\cite{lassila_model_2014,benner2020model,kunisch2002galerkin,ravindran2000reduced,sirovich1987turbulence,willcox2002balanced,christensen1999evaluation,gunzburger2007reduced,caiazzo2014numerical,baiges2013explicit}. The following Theorem \ref{thm:pod_basis} states that the POD basis is optimal in a least-squares sense. The proof is provided by Gubisch and Volkwein in \cite{gubisch2017proper}. 
\begin{theorem}[POD basis]\label{thm:pod_basis}
    Let $Y = [Y_1, \dots, Y_q] := [U_1(t_1), \dots, U_1(t_{r+1}), U_2(t_{1}), \dots, U_M(t_{r+1})] \in \mathbb{R}^{n \times q}$ with $q= M \cdot {\# \text{DoFs}(I_m)}$, $n = \# \text{DoFs}(\mathcal{T}_h)$ and rank $d \leq \min(n,q)$ be the snapshot matrix with a (spatial) column vector for each temporal degree of freedom. Moreover, let $Y = \Psi \Sigma \Phi^T$ be its singular value decomposition with $\Sigma = \text{diag}(\sigma_1, \dots, \sigma_d) \in \R^{d \times d}$ and orthogonal matrices $\Psi = [\psi_1, \dots, \psi_d ] \in \mathbb{R}^{n \times d}$, $\Phi = [\phi_1, \dots, \phi_d ] \in \mathbb{R}^{q \times d}$. Then for $1 \leq N \leq d$ the optimization problem
    \begin{align}\label{eq:pod_optmizaton_problem}
                \min_{\tilde{\psi}_1,\dots,\tilde{\psi}_N\in \mathbb{R}^n} \sum_{j = 1}^q \Big\| Y_j - \sum_{i = 1}^N \left(Y_j, \tilde{\psi}_i\right)_{\mathbb{R}^n}\tilde{\psi}_i \Big\|_{\mathbb{R}^n}^2 \quad \text{s.t.} \quad (\tilde{\psi}_i,\tilde{\psi}_j)_{\mathbb{R}^n} = \delta_{ij}\ \forall 1 \leq i,j \leq N \tag{${\textbf{P}}^N$}
    \end{align}
    where $\lbrace \tilde\psi_i \rbrace_{i = 1}^N \subset \mathbb{R}^n$, and which is being solved by the left-singular vectors $\lbrace \psi_i \rbrace_{i = 1}^N \subset \mathbb{R}^n$ 
    and it holds that
    \begin{align}\label{eq:error_POD}
        \sum_{j = 1}^q \Big\| Y_j - \sum_{i = 1}^N \left(Y_j, {\psi}_i\right)_{\mathbb{R}^n}{\psi}_i \Big\|_{\mathbb{R}^n}^2 = \sum_{i = N+1}^d \sigma_i^2 = \sum_{i = N+1}^d \lambda_i.
    \end{align}
\end{theorem}
Thus, the decay rate of the singular values plays an essential role in the feasibility of the POD approach. If the sum of the squared truncated singular values is sufficiently small for a relatively small $N$, we can utilize a linear combination of a few basis functions $\psi_i$ for a good approximation of elements $Y_j$ living in the high-dimensional FE space.
Although the error of an obtained rank-$N$ approximation can be determined by Equation \eqref{eq:error_POD}, this does not yield an intuitive measure for rank determination. Thus, a widely used criterion to determine the quality of the POD basis heuristically refers to its retained energy or information content $\varepsilon(N)$, cf. \cite{grassle2018pod, gubisch2017proper, lassila_model_2014}. The latter is defined by 
\begin{align} \label{equ:retained_energy}
    \varepsilon(N) = \frac{\sum_\mathrm{i=1}^{N} \sigma_{i}^2}{\sum_{i=1}^{d} \sigma_i^2} = \frac{\sum_{i=1}^{N} \sigma_i^2}{\sum_{i=1}^\mathrm{q} ||{U}_i||^2}. 
\end{align}
Next, the construction of the POD basis is presented. In Algorithm \ref{algo:pod_basis_R_m}, we introduce different approaches depending on the row-to-column ratio of the snapshot matrix. For this, we partly rely on the work of Gr\"aßle et al. in \cite{benner2020model}[Chap. 2].
\begin{algorithm}[H]
    \caption{POD basis generation in $\mathbb{R}^n$} \label{algo:pod_basis_R_m}
    \hspace*{\algorithmicindent} \textbf{Input:} Snapshots $\lbrace Y_j \rbrace_{j = 1}^q \subset \mathbb{R}^n$ and energy threshold $\varepsilon \in [0,1]$.\\
    \hspace*{\algorithmicindent} \textbf{Output:} POD basis $\lbrace \bm{\psi}_i \rbrace_{i = 1}^N\subset \mathbb{R}^n$ and eigenvalues $\lbrace \lambda_i \rbrace_{i = 1}^N$.
    \begin{algorithmic}[1]
        \State Set $Y = [Y_1, \dots, Y_q ] \in \mathbb{R}^{n \times q}$.
        \If{$n \approx q$}
            \State Compute singular value decomposition $[\Psi, \Sigma, \Phi] = \operatorname{SVD}(Y)$.
            \State Compute \scalebox{0.9}{\parbox{1.125\linewidth}{$N = \min\left\lbrace N \in \mathbb{N}\ \middle|\ \varepsilon(N) \geq \varepsilon,\ \ 1 \leq N \leq d \right\rbrace$.}}
            \State Set $\lambda_i = \Sigma_{ii}^2$ and $\bm{\psi}_i = \Psi_{\cdot, i} \in \mathbb{R}^n$ for $1 \leq i \leq N$.
        \ElsIf{$n \ll q$}
            \State Compute eigenvalue decomposition $[\Psi, \Lambda] = \operatorname{Eig}(YY^T)$, where $YY^T \in \mathbb{R}^{n \times n}$.
            \State Compute \scalebox{0.9}{\parbox{1.125\linewidth}{$N = \min\left\lbrace N \in \mathbb{N}\ \middle|\ \varepsilon(N) \geq \varepsilon,\ \ 1 \leq N \leq d \right\rbrace$.}}
            \State Set $\lambda_i = \Lambda_{ii}$ and $\bm{\psi}_i = \Psi_{\cdot, i} \in \mathbb{R}^n$ for $1 \leq i \leq N$.
        \ElsIf{$q \ll n$}
            \State Compute eigenvalue decomposition $[\Phi, \Lambda] = \operatorname{Eig}(Y^TY)$, where $Y^TY \in \mathbb{R}^{q \times q}$.
            \State Compute \scalebox{0.9}{\parbox{1.125\linewidth}{$N = \min\left\lbrace N \in \mathbb{N}\ \middle|\ \varepsilon(N) \geq \varepsilon,\ \ 1 \leq N \leq d \right\rbrace$.}}
            \State Set $\lambda_i = \Lambda_{ii}$ and $\bm{\psi}_i = Y\Phi_{\cdot, i}/\sqrt{\lambda_i} \in \mathbb{R}^n$ for $1 \leq i \leq N$.
        \EndIf
    \end{algorithmic}
\end{algorithm}
\subsection{Tensor-product space-time POD-ROM}
\label{sec:space_time_ROM}
In order to reduce the space-time full-order system (\ref{eq:LES_FOM}) of Section \ref{subsec:tensorproduct_spacetime_FEM_discretization} the general spatial FEM space $V_h$ is replaced by a problem-specific low-dimensional space $V_N = \text{span} \{ \varphi_{N}^1 , \dots, \varphi_{N}^N\}$ obtained by means of POD. This yields the reduced variational formulation:
Find $u_N \in \tilde{X}(\mathcal{T}_k, V_N)$ such that
\begin{align*}
    \tilde{A}(u_N)(\varphi) = \tilde{F}(\varphi) \qquad \forall \varphi \in \tilde{X}(\mathcal{T}_k, V_N).
\end{align*}
The reduced basis matrix can be formed by the concatenation of the reduced basis vectors, viz.
\begin{align} \label{eq:reduced_basis_matrix}
    Z_N = 
    \begin{bmatrix}
        \varphi_{N}^1& \dots & \varphi_{N}^N
    \end{bmatrix} \in \R^{{\# \text{DoFs}(\mathcal{T}_h)} \times N}.
\end{align}
Subsequently, the slabwise discretization for the space-time slab $S_l^n$ with $n-l+1$ time intervals is obtained in analogy to the full-order model of Section \ref{subsection:slabwise_discretization}.
In the case of the heat equation, we utilize the linear equation system described in (\ref{eq:LES_FOM}) and reduce the given matrices in an affine manner. Thus, we arrive at
\begin{align}
    \begin{pmatrix}
        A_N & & &  & \bm{0} \\
        B_N & A_N & & &  \\
         & B_N & A_N &  &  \\
         & & \ddots & \ddots & \\
        \bm{0} & & & B_N & A_N
    \end{pmatrix}
    \begin{pmatrix}
        U_{N_{{l}}} \\
        U_{N_{{l} + 1}} \\
        U_{N_{{l} + 2}} \\
        \vdots \\
        U_{N_{n}}
    \end{pmatrix}
    = 
    \begin{pmatrix}
        F_{N_{{l}}} - B_N U_{N_{{l-1}}} \\
        F_{N_{{l} + 1}} \\
        F_{N_{{l} + 2}} \\
        \vdots \\
        F_{N_{n}}
    \end{pmatrix}
\end{align}
or in brevity
\begin{align}\label{eq:ROM_LES}
    A_{N}  U_{N,{S_l^n}} = F_{N,{S_l^n}}
\end{align}
with the reduced components 
\begin{subequations}
\label{eq:reduced_matrices}
\begin{align}
     A_N &= Z_N^T A Z_N, \\
    B_N &= Z_N^T B Z_N, \\
    F_{N_{i}} &=  Z_N^T F_{{i}}, \quad l \leq i \leq n.   
\end{align}
\end{subequations}
%
\section{A posteriori error-estimator certified reduced-order modeling} \label{sec:error_estimator_certified_ROM}
For further analysis, we consider homogeneous Dirichlet boundary conditions to simplify the presentation, i.e. $u_D = 0$.
Let a goal functional $J: \tilde{X}(\mathcal{T}_k, V(\Omega)) \rightarrow \mathbb{R}$ of the form
\begin{align}
    J(u) = \int_0^T J_1(u(t))\ \mathrm{d}t + J_2(u(T)),
\end{align}
be given, which represents some physical quantity of interest (QoI). Here, $T$ denotes the end time as before.
Now, we want to reduce the difference between the quantity of interest of a fine solution $u^{\text{fine}}$ and a coarse solution $u^{\text{coarse}}$, i.e.,
\begin{align}\label{eq:constrained_optimization_problem}
    J(u^{\text{fine}}) - J(u^{\text{coarse}})
\end{align}
subject to the constraint that the variational formulation of the time-dependent problem (\ref{eq:time_dependent_problem}) is being satisfied. 
Possible choices for the fine and the coarse solution could be $u^{\text{fine}} := u \in X(I, V(\Omega)), u^{\text{coarse}} := u_k \in X_k^{\dG(r)}(\mathcal{T}_k,V(\Omega))$ to control the error caused by the temporal discretization  or $u^{\text{fine}} := u_k \in X_k^{\dG(r)}(\mathcal{T}_k,V(\Omega))$, $ u^{\text{coarse}} := u_{kh} \in X_k^{\dG(r)}(\mathcal{T}_k, V_h)$, with $V_h := V_h^s$ for the heat equation and $V_h := V_h^s = V_h^{s,u} \times V_h^{s,v}$ for the elastodynamics equation, to control the error caused by the spatial discretization.
For more information on space-time error control, we refer the interested reader to \cite{Schmich2009, ThiWi22_arxiv, RoThiKoeWi2022} and for general information on spatial error control to \cite{BeRa96, becker_rannacher_2001, bangerth_rannacher_2003, Endt21}. 
As an extension, in this work we restrict ourselves to the control of the error introduced by reduced-order modeling and thus we consider the full-order-model (FOM) solution $u^{\text{fine}} := u_{kh}^{\text{FOM}} \in X_k^{\dG(r)}(\mathcal{T}_k, V_h^{\text{FOM}})$ as the fine solution, and the reduced-order-model (ROM) solution $ u^{\text{coarse}} := u_{kh}^{\text{ROM}} \in X_k^{\dG(r)}(\mathcal{T}_k, V_h^{\text{ROM}})$ as the coarse solution, with $V_h^{\text{ROM}} \subset V_h^{\text{FOM}} =: V_h$.
First efforts of incorporating the dual-weighted residual (DWR) method in reduced-order modeling have been undertaken by Meyer and Matthies \cite{Meyer2003}, where after computing some snapshots and creating the reduced basis, they used the DWR error estimator to determine which basis vectors have the largest error contribution and only use them for the reduced-order model. 
This can be thought of as a goal-oriented adaptive coarsening of the reduced basis. 
In this work, we focus on another objective, namely the enrichment of the reduced basis depending on the temporal evolution of the quantities of interest. This can be thought of as a goal-oriented adaptive refinement\footnote{In principle coarsening would also be possible, but 
is not the objective in this work. For coarsening, we would need to follow the work of Meyer and Matthies \cite{Meyer2003}.} of the reduced basis, which we  propose to accurately and efficiently compute the solution over the whole temporal domain.

\subsection{Space-time dual-weighted residual method}
\label{sec:ST-DWR}

For the constrained optimization problem (\ref{eq:constrained_optimization_problem}), we define the Lagrange functional for the fine problem as
\begin{align*}
    &\mathcal{L}_{\text{fine}}: X_k^{\dG(r)}(\mathcal{T}_k,V_h^{\text{FOM}}) \times X_k^{\dG(r)}(\mathcal{T}_k,V_h^{\text{FOM}}) \rightarrow \mathbb{R}, \\ 
    &\hspace{3cm}(u^{\text{fine}} , z^{\text{fine}} ) \mapsto J(u^{\text{fine}} ) - \tilde{A}(u^{\text{fine}} )(z^{\text{fine}} ) + \tilde{F}(z^{\text{fine}} ),
\end{align*}
and for the coarse problem as
\begin{align*}
    &\mathcal{L}_{\text{coarse}}: X_k^{\dG(r)}(\mathcal{T}_k,V_h^{\text{ROM}}) \times X_k^{\dG(r)}(\mathcal{T}_k,V_h^{\text{ROM}}) \rightarrow \mathbb{R}, \\
    &\hspace{3cm}(u^{\text{coarse}}, z^{\text{coarse}}) \mapsto J(u^{\text{coarse}}) - \tilde{A}(u^{\text{coarse}})(z^{\text{coarse}}) + \tilde{F}(z^{\text{coarse}}). 
\end{align*}
The stationary points $(u^{\text{fine}} , z^{\text{fine}} )$ and $(u^{\text{coarse}}, z^{\text{coarse}})$ of the Lagrange functionals $\mathcal{L}_{\text{fine}}$ and $\mathcal{L}_{\text{coarse}}$ need to satisfy the Karush-Kuhn-Tucker first-order optimality conditions. Firstly, these stationary points are solutions to the equations
\begin{align*}
    \mathcal{L}^\prime_{\text{fine},z}(u^{\text{fine}} , z^{\text{fine}} )(\delta z^{\text{fine}}) &= 0 \quad \forall \delta z^{\text{fine}} \in X_k^{\dG(r)}(\mathcal{T}_k,V_h^{\text{FOM}}), \\
    \mathcal{L}^\prime_{\text{coarse},z}(u^{\text{coarse}} , z^{\text{coarse}} )(\delta z^{\text{coarse}}) &= 0 \quad \forall \delta z^{\text{coarse}} \in X_k^{\dG(r)}(\mathcal{T}_k,V_h^{\text{ROM}}).
\end{align*}
We call these equations the primal problems and their solutions $u^{\text{fine}}$ and $u^{\text{coarse}}$ the primal solutions.
Secondly, the stationary points must also satisfy the equations
\begin{align*}
    \mathcal{L}^\prime_{\text{fine},u}(u^{\text{fine}} , z^{\text{fine}} )(\delta u^{\text{fine}}) &= 0 \quad \forall \delta u^{\text{fine}} \in X_k^{\dG(r)}(\mathcal{T}_k,V_h^{\text{FOM}}), \\
    \mathcal{L}^\prime_{\text{coarse},u}(u^{\text{coarse}} , z^{\text{coarse}} )(\delta u^{\text{coarse}}) &= 0 \quad \forall \delta u^{\text{coarse}} \in X_k^{\dG(r)}(\mathcal{T}_k,V_h^{\text{ROM}}).
\end{align*}
These equations are called the adjoint or dual problems and their solutions $z^{\text{fine}}$ and $z^{\text{coarse}}$ are the adjoint solutions.

\subsubsection{Primal problem}

Taking the G\^{a}teaux derivatives of the Lagrange functionals $\mathcal{L}_{\text{fine}}$ and $\mathcal{L}_{\text{coarse}}$ with respect to the adjoint solution $z$, we arrive at the primal problem. Since the variational formulation of the PDE is linear in the test functions, we get
{\small
\begin{align*}
    \mathcal{L}^\prime_{\text{fine},z}(u^{\text{fine}} , z^{\text{fine}} )(\delta z^{\text{fine}}) = -\tilde{A}(u^{\text{fine}})(\delta z^{\text{fine}}) + \tilde{F}(\delta z^{\text{fine}}) &= 0 \quad \forall \delta z^{\text{fine}} \in X_k^{\dG(r)}(\mathcal{T}_k,V_h^{\text{FOM}}), \\
   \mathcal{L}^\prime_{\text{coarse},z}(u^{\text{coarse}} , z^{\text{coarse}} )(\delta z^{\text{coarse}}) = -\tilde{A}(u^{\text{coarse}})(\delta z^{\text{coarse}}) + \tilde{F}(\delta z^{\text{coarse}}) &= 0 \quad \forall \delta z^{\text{coarse}} \in X_k^{\dG(r)}(\mathcal{T}_k,V_h^{\text{ROM}}).
\end{align*}
}
We observe that the primal solution can be obtained by solving the original problem, e.g. the heat or the elastodynamics equation, forward in time.

\subsubsection{Adjoint problem}

Taking the G\^{a}teaux derivatives of the Lagrange functionals $\mathcal{L}_{\text{fine}}$ and $\mathcal{L}_{\text{coarse}}$ with respect to the primal solution $u$, we get
\begin{align*}
    \mathcal{L}^\prime_{\text{fine},u}(u^{\text{fine}} , z^{\text{fine}} )(\delta u^{\text{fine}}) = J^\prime_{u}(u^{\text{fine}})(\delta u^{\text{fine}})-\tilde{A}^\prime_{u}(u^{\text{fine}})(\delta u^{\text{fine}},z^{\text{fine}}) &= 0  \\ \forall \delta u^{\text{fine}} \in X_k^{\dG(r)}(\mathcal{T}_k&,V_h^{\text{FOM}}), \\
   \mathcal{L}^\prime_{\text{coarse},u}(u^{\text{coarse}} , z^{\text{coarse}} )(\delta u^{\text{coarse}}) = J^\prime_{u}(u^{\text{coarse}})(\delta u^{\text{coarse}})-\tilde{A}^\prime_{u}(u^{\text{coarse}})(\delta u^{\text{coarse}},z^{\text{coarse}}) &= 0 \\ \forall \delta u^{\text{coarse}} \in X_k^{\dG(r)}(\mathcal{T}_k&,V_h^{\text{ROM}}).
\end{align*}
Hence, to obtain the adjoint solution, we need to solve an additional equation, the adjoint problem
\begin{align}\label{eq:space_time_general_adjoint}
    \tilde{A}^\prime_{u}(u)(\delta u, z) = J^\prime_{u}(u)(\delta  u).
\end{align}
Note that even for nonlinear PDEs and goal functionals the adjoint problem is linear since the semi-linear form in the variational formulation of the PDE is linear in the test functions, however the primal solution enters as it is well-known \cite{becker_rannacher_2001}.

\begin{remark}\label{remark:adjoint_linear_problem}
    For linear PDEs, like the heat or the elastodynamics equation, the left-hand side of the adjoint problem (\ref{eq:space_time_general_adjoint}) simplifies to
    \begin{align*}
        \tilde{A}^\prime_{u}(u)(\delta u, z) = \tilde{A}(\delta u)(z).
    \end{align*}
    For linear goal functionals, like the mean-value functional, the right-hand side of the adjoint problem (\ref{eq:space_time_general_adjoint}) reduces to
    \begin{align*}
        J^\prime_{u}(u)(\delta u) = J(\delta  u).
    \end{align*}
    In particular for a linear problem, i.e. linear PDE and goal functional, we have the adjoint problem
\begin{align}\label{eq:space_time_linear_adjoint}
        \tilde{A}(\delta u)(z) = J(\delta  u),
    \end{align}
    which does not depend on the primal solution $u$ anymore.
\end{remark}

By Remark (\ref{remark:adjoint_linear_problem}), the adjoint problem for the heat equation reads
\begin{align*}
    &\tilde{A}(\delta u)(z) =  J^\prime_{u}(u)(\delta  u) \\
    \Leftrightarrow
    \sum_{m = 1}^M \int_{I_m} (\partial_t \delta u,z)  + (\nabla_x \delta u,\nabla_x z)\ \mathrm{d}t &+ \sum_{m = 1}^{M-1}([\delta u]_m, z_{m}^{+}) + (\delta u_{0}^+,z_{0}^{+}) =  J^\prime_{u}(u)(\delta  u).
\end{align*}
We now use integration by parts in time to move the time derivative from the test function $\delta u$ to the adjoint solution $z$ and we get
\begin{align*}
    \sum_{m = 1}^M \int_{I_m} (\delta u,-\partial_t z)  + (\nabla_x \delta u,\nabla_x z)\ \mathrm{d}t - \sum_{m = 1}^{M-1}(\delta u_{m}^{-}, [z]_{m}) + (\delta u_M^-,z_M^{-}) =  J^\prime_{u}(u)(\delta  u).
\end{align*}
For the elastodynamics equation the adjoint problem can be derived in a similar fashion as
\begin{align*}
    &\sum_{m = 1}^M \int_{I_m}(\delta v, -\partial_t z^u) + (\sigma(\delta u), \nabla_x z^u) + (\delta u, -\partial_t z^v) - (\delta v, z^v) \ \mathrm{d}t \\
    &\hspace{1cm}- \sum_{m=1}^{M-1} \left( (\delta v^-_{m}, [z^{u}]_m) + (\delta u^-_{m}, [z^{v}]_m) \right) + (\delta v_{M}^-, z_{M}^{u,-}) + (\delta u_{M}^-, z_{M}^{v,-}) = J^\prime_{U}(U)(\delta  U).
\end{align*}
We notice that the adjoint problem now runs backward in time.

\subsubsection{Error identity and temporal localization for linear problems}

For the sake of simplicity, we assume that we are dealing with a linear PDE and goal functional.
Then we have the error identity
\begin{align}\label{eq:error_identity}
     J(u^{\text{fine}}) - J(u^{\text{coarse}}) = -\tilde{A}(u^{\text{coarse}})(z^{\text{fine}}) + \tilde{F}(z^{\text{fine}}) =: \eta.
\end{align}
The proof relies on both the linearity of the goal functional and the PDE, and the definition of the adjoint and primal problems:
\begin{align*}
    J(u^{\text{fine}}) - J(u^{\text{coarse}}) = J(u^{\text{fine}} - u^{\text{coarse}}) = \tilde{A}(u^{\text{fine}} - u^{\text{coarse}})(z^{\text{fine}})  = -\tilde{A}(u^{\text{coarse}})(z^{\text{fine}}) + \tilde{F}(z^{\text{fine}}).
\end{align*}
In the DWR literature for spatial and temporal discretization error control this kind of error identity (\ref{eq:error_identity}) would be useless, because for most applications $z^{\text{fine}}$ is the analytical solution which is not known a priori and replacing it by $z^{\text{coarse}}$ yields bad error estimates. Thus, for FEM discretization error control the dual weights $z^{\text{fine}} - z^{\text{coarse}}$ are being used, which can be approximated by post-processing of the dual solution. However, in our case $z^{\text{fine}} := z_{kh}^{\text{FOM}} \in X_k^{\dG(r)}(\mathcal{T}_k,V_h^{\text{FOM}})$ is the full-order-model dual solution, which is computable but comes with an expense. Moreover, in our numerical experiments we will observe that using a reduced-order-model dual solution $z^{\text{coarse}} := z_{kh}^{\text{ROM}} \in X_k^{\dG(r)}(\mathcal{T}_k,\tilde{V}_h^{\text{ROM}})$ still produces excellent error estimates for our problems if the dual reduced basis is sufficiently large. We point out that the dual spatial reduced-order-model function space $\tilde{V}_h^{\text{ROM}}$ needs to differ from the primal spatial reduced-order-model function space $V_h^{\text{ROM}}$ if we want to capture the dynamics of the dual problem and want to have a non-zero error estimator.

To localize the error in time, we just need to assemble the primal residual (\ref{eq:error_identity}) slabwise. In particular, to localize the error to each time interval $I_m$, we simply need to assemble the primal residual on each time interval separately. More concretely, for the heat equation the error on the time interval can be computed from the primal linear equation system, the coarse primal solution and the fine dual solution by
\begin{align}\label{eq:error_estimator_heat}
    \eta\restrict{I_m} = \sum_{i=1}^{\# \text{DoFs}(I_m)}\left\{(Z_m^{\text{fine}})^T\left(-AU_m^{\text{coarse}} + F_m - BU_{m-1}^{\text{coarse}}\right)\right\}_i.
\end{align}
The error estimator on the time interval $I_m$ for elastodynamics can be derived analogously by using the linear system (\ref{eq:elasto_dG_linear_system}) of the primal problem. 

To test whether we need to use the fine dual solution for our error estimates or whether we can replace it with a coarse dual solution, we use the effectivity index as a measure of the quality of our error estimator. The effectivity index is the ratio of the estimated and the true errors, i.e.
\begin{align}\label{eq:effectivity_index}
    \Ieff := \left|\frac{\eta}{J(u^{\text{fine}}) - J(u^{\text{coarse}})}\right|.
\end{align}
We desire $\Ieff \approx 1$, since then the error estimator can reliably predict the reduced-order-modeling error and we also observe this in the numerical tests in Section \ref{sec:numerical_tests}.

\subsubsection{Space-time dual-weighted residual method for nonlinear problems}

For nonlinear problems, like the heat equation with nonlinear goal functional in Section \ref{sec:2d_heat}, we do not have an error identity  anymore as in (\ref{eq:error_identity}) for the linear case. Based on the proof in \cite{becker_rannacher_2001}[Proposition 2.3], we have the following error representation formula.
\begin{theorem}[Error representation for nonlinear problems]
\begin{align*}
    J(u^{\text{fine}}) - J(u^{\text{coarse}}) = -\tilde{A}(u^{\text{coarse}})(z^{\text{fine}}) + \tilde{F}(z^{\text{fine}}) + R,
\end{align*}
with the quadratic remainder term
{\small
\begin{align*}
    R = \int_0^1 \Big[ &\tilde{A}^{\prime\prime}_{uu}(u^{\text{coarse}} + s(u^{\text{fine}}-u^{\text{coarse}}))(u^{\text{fine}}-u^{\text{coarse}}, u^{\text{fine}}-u^{\text{coarse}}, z^{\text{fine}}) \\ &- J^{\prime\prime}_{uu}(u^{\text{coarse}} + s(u^{\text{fine}}-u^{\text{coarse}}))(u^{\text{fine}}-u^{\text{coarse}}, u^{\text{fine}}-u^{\text{coarse}})\Big] \cdot s\ \mathrm{d}s .
\end{align*}
}
\end{theorem}
\begin{proof}
    In the following we will show that $R = J(u^{\text{fine}}) - J(u^{\text{coarse}}) + \tilde{A}(u^{\text{coarse}})(z^{\text{fine}}) - \tilde{F}(z^{\text{fine}})$ holds. For abbreviation, we use the notation $u := u^{\text{fine}}$, $\tilde{u} := u^{\text{coarse}}$ and $z := z^{\text{fine}}$. Then, using integration by parts we get
    {\footnotesize
    \begin{align*}
        R &= \int_0^1 \left[ \tilde{A}^{\prime\prime}_{uu}(\tilde{u} + s(u-\tilde{u}))(u-\tilde{u}, u-\tilde{u}, z) - J^{\prime\prime}_{uu}(\tilde{u} + s(u-\tilde{u}))(u-\tilde{u}, u-\tilde{u})\right] \cdot s\ \mathrm{d}s \\
        &= -\int_0^1 \left[ \tilde{A}^{\prime}_{u}(\tilde{u} + s(u-\tilde{u}))(u-\tilde{u}, z) - J^{\prime}_{u}(\tilde{u} + s(u-\tilde{u}))(u-\tilde{u})\right] \cdot 1\ \mathrm{d}s + \left[ \tilde{A}^{\prime}_{u}(u)(u-\tilde{u}, z) - J^{\prime}_{u}(u)(u-\tilde{u})\right] \cdot 1 - 0 .
    \end{align*}
    }
    We observe that $\tilde{A}^{\prime}_{u}(u)(u-\tilde{u}, z) - \tilde{J}^{\prime}_{u}(u)(u-\tilde{u}) = 0$, since  $z := z^{\text{fine}}$ is the solution of the fine dual problem. Thus, by the fundamental theorem of calculus and $\tilde{A}(u)(z) = \tilde{F}(z)$, we have
    \begin{align*}
        R &= - \left[ \tilde{A}(u)(z) - J(u) - \tilde{A}(\tilde{u})(z) + J(\tilde{u}) \right] = J(u) - J(\tilde{u}) + \tilde{A}(\tilde{u})(z) - \tilde{F}(z).
    \end{align*}
    This completes the proof.
\end{proof}
To make the error estimator computable, we neglect the quadratic remainder term and arrive at the same primal error estimator  (\ref{eq:error_identity}) as for linear problems
\begin{align*}
    \eta := -\tilde{A}(u^{\text{coarse}})(z^{\text{fine}}) + \tilde{F}(z^{\text{fine}}).
\end{align*}
Similarly as before, we replace the full-order dual solution $z^{\text{fine}}$ in the error estimator with a reduced-order-model dual solution $z^{\text{coarse}}$. Note that due to these approximations, the effectivity index for nonlinear problems is expected not to be close to 1. Clearly, for highly nonlinear problems (e.g., quasi-linear or fully nonlinear) and nonlinear goal functionals, both estimator parts are necessary as demonstrated in \cite{EndtLaWi20}[Figure 4] and \cite{EndtLaWi18}[Sec. 6.5]. However, in our numerical tests, we see that the estimated error still yields a reasonable approximation to the true error.

\subsection{Error estimator based ROM updates}
In this section, we present our novel approach of a goal-oriented incremental reduced-order model. 
In the MORe DWR method, we marry a reduced-order model with a DWR-based error estimator and an incremental version of the POD algorithm.
The MORe DWR method addresses the problems that occur when a reduced-order model has to deal with solution behavior that is not already captured and incorporated during basis generation. In general, this yields an increasing error between full- and reduced-order solutions.
Thus, the presented approach aims to detect changes in solution behavior, or more precisely, differences in the evaluated quantities of interest by means of the full or reduced model during the temporal evolution. If the error increases to intolerable heights, the method allows an adaptive on-the-fly basis enrichment with snapshots of the new behavior. Hence, the reduced model can be incrementally modified until the error is sufficiently small. 

In more detail, we rely on the space-time reduced-order model presented in Section \ref{sec:space_time_ROM} and apply our findings on error control of Section \ref{sec:ST-DWR}. The use of an error estimate rather than an analytical error bound entails practical advantages since its application is more versatile and we can use the method even if no error bounds are known. Further, an incremental basis generation is mandatory for the method to reduce computational operations and thus to be fast. The incremental SVD satisfies these requirements and allows an update only requiring the prior SVD and the new snapshots. The incremental SVD is presented in Section \ref{sec:additive_svd}. In this context, we also introduce the incremental POD as a trimmed version of the incremental SVD. 
Subsequently, the overall MORe DWR framework is depicted in Section \ref{sec:incremental_ROM}. Here, all the ingredients are assembled and the final algorithm is presented.

In summary, our novel approach neglects a computationally heavy offline phase and directly solves the reduced model. Full-order solves are only required for the basis enrichment and are held to a minimum. Moreover, the reduced evaluation of the quantity of interest can be certified.

\subsubsection{Incremental Proper Orthogonal Decomposition}\label{sec:iPOD}
\label{sec:additive_svd}
This section aims to derive an algorithm that updates an already existing truncated SVD (tSVD) or solely its left-singular (POD) vectors according to modifications of the snapshot matrix without recomputing the whole tSVD or requiring access to the snapshot matrix.
This methodology can then be used to update the POD incrementally by appending additional snapshots to the snapshot matrix. For this purpose, we rely on the general approach of an additive rank-b modification of the SVD, mainly developed by \cite{brand2002incremental, brand2006fast} and applied to the model-order reduction of fluid flows in \cite{kuhl2023incremental}. Although this approach provides a variety of possible modifications, e.g. resizing of the matrix, modification of individual values, or exchanging rows and columns, we are merely interested in the updates of columns, i.e. adding columns to the matrix, and thus restrict the proceeding on this. The following steps are based on \cite{kuhl2023incremental}[Section 2.2].

We start with a given snapshot matrix ${Y} \in \mathbb{R}^{{\# \text{DoFs}(\mathcal{T}_h)} \times \mathrm{\Tilde{m}}}$ that includes $\Tilde{m}>0$ snapshots. Usually, $\Tilde{m}$ is equal or connected to the number of already computed time steps. Further, we have the rank-$N$ tSVD ${U} {S} {V}^\mathrm{T}$ of the matrix $Y$. Additionally, let $b\in\mathbb{N}$ newly computed snapshots $\{U_1, \dots, U_b\}$ be stored in the bunch matrix
\begin{align} \label{eq:bunch_matrix}
    B = \begin{bmatrix}
        u_1& \dots &u_b
\end{bmatrix} \in \R^{{\# \text{DoFs}(\mathcal{T}_h)} \times b}.
\end{align}
We now aim to compute the tSVD that is updated by the information contained in the bunch matrix $B$ according to 
\begin{align*}
 \Tilde{U} \Tilde{S} \Tilde{V}^{T} = 
    \Tilde{Y} = 
\begin{bmatrix}
        Y & B
\end{bmatrix}
\end{align*}
without explicitly recomputing $Y$ or $\Tilde{Y}$ due to performance and memory reasons which was the original motivation of Brand's work on the incremental SVD, cf. \cite{brand2002incremental,brand2006fast}.

Therefore, we write the column update as an additive operation given as
\begin{align}
    \begin{bmatrix}
        Y & B
\end{bmatrix}
= 
\begin{bmatrix}
        Y & 0_{{\# \text{DoFs}(\mathcal{T}_h)} \times b}
\end{bmatrix}
+ 
B
\begin{bmatrix}
        0_{b \times {\Tilde{m}}} &
        I_{b \times b}
\end{bmatrix}
\end{align}
to apply the additive rank-b modification to the SVD according to \cite{kuhl2023incremental} and obtain the rank-$\Tilde{N}$ tSVD of $\Tilde{Y}$ with $\Tilde{N} \leq N+b$ and
\begin{align}
    \Tilde{{V}} &= 
    \begin{bmatrix}
            {V} & {0} \\
            {0} & {I}
    \end{bmatrix} {V}^\prime(:,1:\Tilde{N}) \label{equ:update_V_column} \\
    \Tilde{{S}} &= {S}^\prime(1:\Tilde{N},1:\Tilde{N}) \label{equ:update_S_column} \\
    \Tilde{{U}} &= 
    \begin{bmatrix}
            {U} & {Q}_\mathrm{B}
    \end{bmatrix} {U}^\prime(:,1:\Tilde{N}) \label{equ:update_U_column} \, ,
\end{align}
where $F = {U}^\prime {S}^\prime {V}^{\prime^T} \in \R^{N+b \times N+b}$ is the SVD of 
\begin{align}
    {F} = 
    \begin{bmatrix}
            {\Sigma} & {U}^\mathrm{T} {B} \\ {0} & {R}_\mathrm{B}
    \end{bmatrix}
    \label{equ:matrix_K_column}
\end{align}
and $Q_B \in \R^{\# \text{DoFs}(\mathcal{T}_h) \times b}$ and $R_B \in \R^{b \times b}$ are given by the QR decomposition
\begin{align}
    {Q}_\mathrm{B} {R}_\mathrm{B} = ({I} - {U}{U}^\mathrm{T}){B} \in \R^{\# \text{DoFs}(\mathcal{T}_h) \times b}.
\end{align}
For the POD basis update, we identify $U$ and $\Tilde{U}$ with the previous and updated versions of the reduced basis matrix $Z_N$ including the POD vectors, respectively. We also neglect the update of the right-singular vectors in (\ref{equ:update_V_column}), since it does not provide any additional benefit apart from extra computational effort for the reduced-order model, cf. Theorem \ref{thm:pod_basis}. The singular values are considered for the rank determination but they come within zero computational cost.
In conclusion, (\ref{equ:update_S_column})-(\ref{equ:matrix_K_column}) serve as the basis for the on-the-fly or incrementally computed POD (iPOD) in this paper.

An additional technical observation: For bunch matrices with small column rank $b$, the iPOD algorithm is invoked frequently, and algebraic subspace rotations possibly involved do not preserve orthogonality, cf. \cite{brand2006fast,fareed2018incremental,bach2019randomized,fareed2019note,zhang2022answer}. Hence, a numerically induced loss of orthogonality of the POD basis vectors can occur. In order to deal with this problem an additional orthonormalization of 
$\begin{bmatrix}
    {U} & {Q}_\mathrm{B}
\end{bmatrix}$
is recommended. Algorithm \ref{algo:iPOD} drafts the implementation of an incremental POD update. Here, $Z_N$ and ${\Sigma = [\sigma_1, \dots, \sigma_N] \in \R^N}$ denote the reduced basis matrix of (\ref{eq:reduced_basis_matrix}) and its respective singular values. In addition, the bunch matrix $B$ introduced in (\ref{eq:bunch_matrix}) including $b$ snapshots is used as an input. The information content captured by the reduced basis is determined by the energy threshold $\varepsilon$.
\begin{algorithm}[H]
    \caption{Incremental POD update} \label{algo:iPOD}
    \hspace*{\algorithmicindent} \textbf{Input:} Reduced basis matrix $Z_N \in \R^{{\# \text{DoFs}(\mathcal{T}_h)} \times N}$,
    singular value vector ${\Sigma = [\sigma_1, \dots, \sigma_N] \in \R^N}$,
    bunch matrix $B \in \R^{{\# \text{DoFs}(\mathcal{T}_h)} \times b}$,     
    and energy threshold $\varepsilon \in [0,1]$.\\
    \hspace*{\algorithmicindent} \textbf{Output:} Reduced basis matrix $Z_N \in \R^{{\# \text{DoFs}(\mathcal{T}_h)} \times \Tilde{N}}$,
    singular value vector ${\Sigma = [\sigma_1, \dots, \sigma_{\Tilde{N}} ] \in \R^{\Tilde{N}}}$ 
    \begin{algorithmic}[1]
        \State $H = Z_N^T B$
        \State $P = B - Z_N H$
        \State $[Q_P, \, R_P] = \text{QR}(P)$
        \State $Q = [Z_N \; Q_P] $
        \State ${F} = \begin{bmatrix}
            {\text{diag}(\Sigma)} & {H} \\
            {0} & {R}_{P}
        \end{bmatrix}$
        \If{Q not orthogonal}
            \State $[Q,\, R] = \text{QR}(Q)$
            \State $ F = RF$
        \EndIf
        \State $[U^\prime, \Sigma^\prime] = \text{SVD}(F) $
        \State $\Tilde{N} = \min\left\lbrace N \in \mathbb{N}\ \middle|\ \varepsilon(N) \geq \varepsilon,\ \ 1 \leq N \leq d \right\rbrace$
        \State $\Sigma = \text{diag}(\Sigma^\prime)(1:\Tilde{N})$ 
        \State $Z_N = Q U^\prime(:,1:\Tilde{N})$ 
    \end{algorithmic}
\end{algorithm}
Note that checking if the orthogonality is preserved can be computationally expensive. Thus, we resort to a heuristic approach by sole validation if the first and last columns of a matrix are orthogonal. 

\subsubsection{Goal-oriented certified incremental ROM}\label{sec:incremental_ROM}
In this section, we assemble the space-time ROM presented in Section \ref{sec:space_time_ROM} and the incremental POD of Section \ref{sec:iPOD} with the findings on goal-oriented error control of Section \ref{sec:ST-DWR}. This yields an adaptive goal-oriented incremental reduced-order model.
Firstly, next to the slab definition we introduce the parent-slab notion as a further decomposition of the space-time domain. A parent-slab unifies several slabs that are consecutive in time and is defined as
\begin{align*}
    P_k^r = \{ S_l^n \; | \; l \geq k \; \land \; n \leq r \}.
\end{align*}
Now, our approach is designed to work without any prior knowledge or exploration of the solution manifold while also attempting to minimize the full-order operations. Thus, we aim to solve the reduced-order model parent-slab wise and --if necessary-- adaptively enrich the reduced basis by means of the iPOD with full-order solutions of the parent-slab until the reduced basis is good enough to meet a given estimated error tolerance for the chosen cost functional. For this, we identify the fine and coarse solutions introduced in the DWR method with the finite element and reduced basis solutions, respectively, and estimate the error on each slab of the parent-slab. The full-order solution used for the basis enrichment is computed on the slab where the error is the largest. We remark that both the primal and dual full-order solutions are computed on this slab and are used to enrich the primal and dual bases. So, for each enrichment two full-order solves are conducted. After having finished this iterative process on a parent-slab, the obtained basis is transferred to the proceeding parent-slab and is used as a starting point to solve the reduced-order model where the whole procedure is repeated. So if the solution behavior on the next parent-slab only differs slightly from the already observed behavior, the reduced basis at hand should be able to reproduce most of the behavior. Thus, few basis updates would be sufficient such that a fast computation of the reduced solution can be expected. However, if the solution behavior changes drastically the error estimate will detect this and further refinements of the basis will be conducted to ensure that the solution meets the error tolerance. 
We observe that this procedure is perfectly compatible with the adaptive basis selection based on DWR estimates presented by Meyer and Matthies in \cite{Meyer2003} to reduce the dimension of the reduced space. Thus, if incorporated it would be possible to either enrich or delude the reduced basis adjusted to the problem statement.

The resulting approach is outlined in Algorithm \ref{algo:iROM}. 
For the sake of simplicity, we decompose the space-time cylinder in $K$ parent-slabs
of fixed length $L$ and enumerate them with respect to time, viz. $P_1, P_2, \dots, P_K$. In order to identify the affiliation of a slab to a parent-slab $P_k$, the slabs it contains are denoted by $S_{P_k}^1, S_{P_k}^2, \dots, S_{P_k}^L$ with $1 \leq k \leq K$. The discretized primal systems for each slab $S_{P_k}^j$ are expressed in (\ref{eq:LES_heat_fom}) and (\ref{eq:ROM_LES}) for the full- and reduced-order models, respectively. 
For the dual problem
\begin{align}
    A^\prime Z_{S_{P_k}^{l}} &= J_{S_{P_k}^{l}} \quad \text{and} \label{eq:discretized_adj_FOM} \\
    A^\prime_{N} Z_{N,{S_{P_k}^l}} &= J_{N,{S_{P_k}^l}} \label{eq:discretized_adj_ROM}
\end{align}
denote the discretized full- and reduced-order systems of the adjoint problem (\ref{eq:space_time_general_adjoint}). Further, the evaluation of the error estimator (\ref{eq:error_estimator_heat}) on slab $S_{P_k}^l$ is given by $\eta_{N,{S_{P_k}^l}} \left(U_{N,{S_{P_k}^l}} ,Z_{N,{S_{P_k}^l}}\right)$. Note that the reduced primal and dual solutions are deployed to enable an evaluation independent of the full-order system size and thus a fast error evaluation. Lastly, the incremental POD (\ref{algo:iPOD}) is referred to by the abbreviation iPOD with the reduced basis, new snapshots bundled in the snapshot matrix, and singular values as input and the new POD basis as output.

\begin{algorithm}[H]
    \caption{Incremental ROM} \label{algo:iROM}
    \hspace*{\algorithmicindent} \textbf{Input:} Initial condition $U_0:=U(t_0)$, primal and dual reduced basis matrices $Z^p_N$ and $Z^d_N$, energy threshold $\varepsilon \in [0,1]$ and error tolerance $\text{tol}>0$.\\
    \hspace*{\algorithmicindent} \textbf{Output:} Primal and dual reduced basis matrices $Z^p_N$ and $Z^d_N$ and reduced primal solutions $U_{N,I_m}$ for all $1\leq m \leq M$.
    \begin{algorithmic}[1]
        \For{$k = 1, 2, \dots, K$}
            \While{$\eta_{max} > tol$}
                \For{$l=1, 2, \dots, L$}
                    \State Solve reduced primal system (\ref{eq:ROM_LES}): $A_{N} U_{N,{S_{P_k}^l}} = F_{N,{S_{P_k}^l}}$
                \EndFor
                \For{$l=L, L-1, \dots, 1$}
                    \State Solve reduced dual system (\ref{eq:discretized_adj_ROM}): $A^\prime_{N} Z_{N,{S_{P_k}^l}} = J_{N,{S_{P_k}^l}}$ 
                \EndFor
                \For{$l=1, 2, \dots, L$}
                    \State Compute error estimate: $\eta_{N,{S_{P_k}^l}} \left(U_{N,{S_{P_k}^l}} ,Z_{N,{S_{P_k}^l}}\right)$
                \EndFor
                \State $\eta_{max} = \max\limits_{1 \leq l \leq L} \left| \eta_{N,{S_{P_k}^l}} \right|$
                \If{$\eta_{max} > tol$}
                    \State ${l_{max}} =  \argmax\limits_{1 \leq l \leq L} \left| \eta_{N,{S_{P_k}^l}} \right|$
                    \vspace{0.2em}
                    \State Solve primal full-order system (\ref{eq:LES_heat_fom}): $A U_{S_{P_k}^{l_{max}}} = F_{S_{P_k}^{l_{max}}}$
                    \vspace{0.2em}
                    \State Update primal reduced basis: $Z^p_N = \text{iPOD}(Z^p_N, [U_{S_{P_k}^{l_{max}}}(t_1), \dots ,U_{S_{P_k}^{l_{max}}}(t_{r+1})], \Sigma)$
                    \vspace{0.2em}
                    \State Solve dual full-order system (\ref{eq:discretized_adj_FOM}): $A^\prime Z_{S_{P_k}^{l_{max}}} = J_{S_{P_k}^{l_{max}}}$
                    \vspace{0.2em}
                    \State Update dual reduced basis: $Z^d_N = \text{iPOD}(Z^d_N, [Z_{S_{P_k}^{l_{max}}}(t_1), \dots ,Z_{S_{P_k}^{l_{max}}}(t_{r+1})], \Sigma)$
                    \vspace{0.2em}
                    \State Update reduced system components and error estimator w.r.t (\ref{eq:reduced_matrices})
                \EndIf
            \EndWhile
        \EndFor
        \vspace{0.5cm}
        \State ----------- Validation loop ----------- \Comment{This is an optional validation mechanism of the model.}
        \vspace{0.5cm}
        \For{$k = 1, 2, \dots, K$}
            \For{$l=1, 2, \dots, L$}
                    \State Solve primal reduced system: $A_{N} U_{N,{S_{P_k}^l}} = F_{N,{S_{P_k}^l}}$
            \EndFor
        \EndFor
        \For{$k = K, K-1, \dots, 1$}
            \For{$l=L, L-1, \dots, 1$}
                 \State Solve dual reduced system: $A^\star_{N} Z_{N,{S_{P_k}^l}} = J_{N,{S_{P_k}^l}}$ 
            \EndFor
        \EndFor
        \For{$k = 1, 2, \dots, K$}
           \For{$l=1, 2, \dots, L$}
                \State Compute slab estimate: $\eta_{N,{S_{P_k}^l}} (U_{N,{S_{P_k}^l}} ,Z_{N,{S_{P_k}^l}})$
            \EndFor
        \EndFor
    \end{algorithmic}
\end{algorithm}

In addition to the previously mentioned steps, we add an optional validation loop whose purpose depends on the application. Specifically, it consists in recomputing the whole reduced solutions with the final reduced basis and evaluating its error again. If the generated reduced basis is meant to be reused, the additional validation of its accuracy ensures that the reduced basis is well suited to approximate the solution for the whole time domain. This is mainly the case in an optimization process or if the MORe DWR method is used for manifold exploration.
However, if the only purpose is a one-time evaluation of a quantity of interest, the validation can be neglected for performance reasons.

Furthermore, we note that similar to the mere approximation error of the POD in (\ref{eq:error_POD}) the physical interpretation of the error estimate is not intuitive. Therefore, we are considering a relative measurement of the approximation quality. However, the full-order solutions are not available for a normalization of the error so that we resort to 
\begin{align*}
    J\left(  U_{{S_{P_k}^l}}  \right) \approx J\left( U_{N,{S_{P_k}^l}} \right) + \eta_{N,{S_{P_k}^l}}.
\end{align*}
This results in the relative error estimator $\eta^{{rel}}_{N,{S_{P_k}^l}}$ on slab ${S_{P_k}^l}$ defined by
\begin{align}\label{eq:relative_error_estimator}
    \eta^{{rel}}_{N,{S_{P_k}^l}}
    = \frac{\eta_{N,{S_{P_k}^l}}}{J\left(u_{S_{P_k}^l}\right)}
    \approx \frac{\eta_{N,{S_{P_k}^l}}}{J\left(u_{N,{S_{P_k}^l}}\right) + \eta_{N,{S_{P_k}^l}} }.
\end{align}
%

\section{Numerical tests}
\label{sec:numerical_tests}
In order to demonstrate our methodology, we perform numerical tests on three different problem configurations. For the first two numerical tests, we perform computations for the heat equation in 1+1D and 2+1D. For the former, we use a linear goal functional and for the latter, we use a nonlinear goal functional. To demonstrate the flexibility of our temporal discretization, we use Gauss-Legendre quadrature points in time for the heat equation and a $\dG(1)$ time discretization.
As the third problem configuration, we consider a 3+1D cantilever beam as a benchmark problem from elastodynamics. For this problem, we use Gauss-Lobatto quadrature points in time, which are the support points for conventional time-stepping schemes, and we use a $\dG(2)$ time discretization.

All our computations have been performed on a personal computer with an Intel i5-7600K CPU @ 3.80GHz × 4 and 16GB of RAM. The space-time FEM codes have been written in deal.II \cite{dealii2019design, dealII94} and the reduced-order modeling has been performed with NumPy \cite{NumPy2020} and SciPy \cite{SciPy2020}. The data between the codes is exchanged via the hard disk.

\subsection{1+1D Heat equation} \label{subsec:1+1d_heat_numerical}
For our first numerical test, we construct a 1+1D heat equation problem; see Formulation \ref{form:variational_heat}. We consider the spatial domain $\Omega = (0,1)$ and the temporal domain $I = (0,4)$. We use a single moving heat source that changes its temperature after each second and moves through the spatial domain with a heating interval width of $0.1$ from $x = 0.1$ to $x=0.9$ and then back to $x=0.1$. For this, we use the right-hand side function
\begin{align*}
    f(t, x) := \begin{cases}
        0.2 & t \in (0,1),\, -0.05 \leq x -0.4t -0.1 \leq 0.05,\\
        -0.5 & t \in (1,2),\, -0.05 \leq x -0.4t -0.1 \leq 0.05,\\
        1.0 & t \in (2,3),\, -0.05 \leq x +0.4(t-2) -0.9 \leq 0.05,\\
        -0.75 & t \in (3,4),\, -0.05 \leq x +0.4(t-2) -0.9 \leq 0.05.
    \end{cases}
\end{align*}
We use a zero initial condition, homogeneous Dirichlet boundary conditions, and the time-averaged mean value goal functional $J(u) := \frac{1}{4}\int_0^4\int_0^\frac{1}{2} u(t,x)\ \mathrm{d}x\ \mathrm{d}t$. We point out that the goal functional does not have support on the entire spatial domain, but only on its lower half $(0, \frac{1}{2}) \subsetneq \Omega$.

For the reduced-order model, we choose that the primal and dual reduced bases have to preserve $\varepsilon = 1 - 10^{-8}$ of the information. As previously stated, we resort to the relative error estimate $\eta^{{rel}}_{N,{S_{P_k}^l}}$ developed in (\ref{eq:relative_error_estimator}) and allow errors up to a tolerance of $1\%$. We consider this to be a reasonable tolerance for many applications. The full-order model is characterized by $n = 8,193$ and $q = 10,240$ DoFs in space and time, respectively. This gives us a total of $n\cdot q = 83,896,320$ space-time degrees of freedom. Further, the temporal domain is split up into $M = 5,120$ time slabs. For the incremental ROM, we choose a total amount of $K = 64$ parent-slabs on which the slabs are evenly distributed, i.e. $L = 80$. 

In Figure \ref{fig:reduced_solution_1d_heat}, we display the full-order space-time solution $u_h$ as well as the true error $u_h - u_N$ between the full-order space-time solution $u_h$ and the reduced-order space-time solution $u_N$ obtained using MORe DWR. Looking at the error,
we observe that the reduced-order model becomes less accurate for $x \in (\frac{1}{2},1)$ than for $x \in (0, \frac{1}{2})$, which is the spatial support of the goal functional.
This shows that our incremental POD is goal-oriented.
\begin{figure}[H]
    \centering
    \subfloat[space-time full-order solution $u_h$]{
    \includegraphics[width=0.8\textwidth]{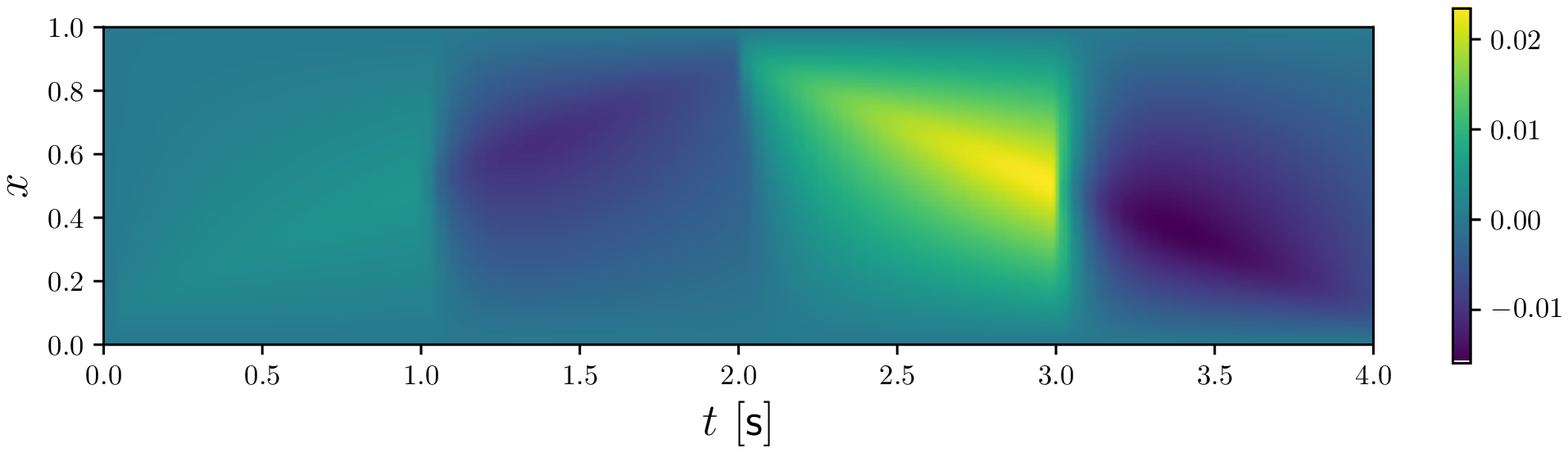}}%
    \\
     \subfloat[space-time true error $u_h - u_N$]{
    \includegraphics[width=0.8\textwidth]{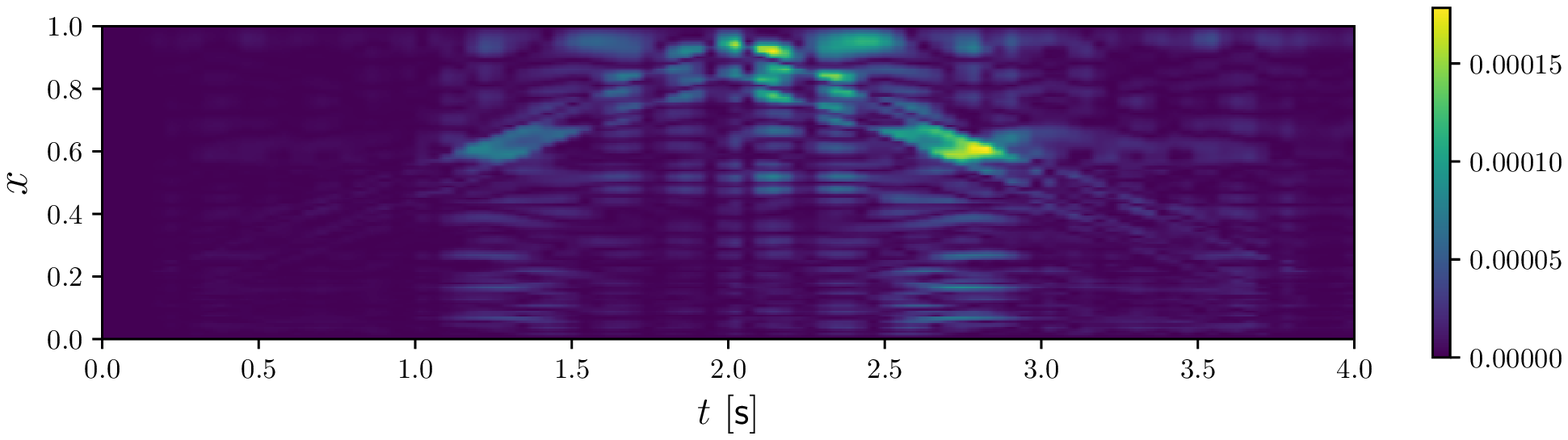}}%
    \caption{Space-time solution and error for the 1+1D heat equation.}
    \label{fig:reduced_solution_1d_heat}
\end{figure}
In Figure \ref{fig:cost_functional_1d_heat}, we compare the time trajectories of the goal functional restricted to each time slab for the full-order space-time solution $u_h$ and the reduced-order space-time solution $u_N$. 
It illustrates that both trajectories are not distinguishable 
from each other indicating that the reduced-order model captures the temporal evolution of the quantity of interest accurately even with changing solution behavior. This good approximation quality can also be observed when regarding the time-averaged cost functional. We obtain 
$J(u_h) = 2.0608 \cdot 10^{-4}$ and $J(u_N) = 2.0583\cdot 10^{-4}$ yielding a relative error of $\eta_{max} = 0.1210\%$.

\begin{figure}[H]
    \centering
    \includegraphics[width=0.65\textwidth]{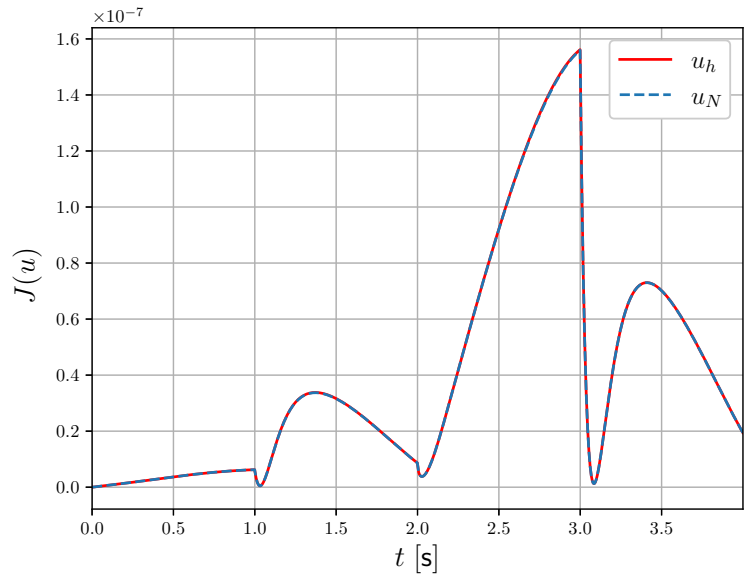}
    \caption{Temporal evolution of cost functional for the 1+1D heat equation. }
    \label{fig:cost_functional_1d_heat}
\end{figure}
We compare the temporal error estimate with the exact temporal error on each slab in Figure~\ref{fig:error_estimator_1d_heat}. The general tendencies of both curves are similar. The exact error is on average more than one magnitude smaller than the error tolerance of $1\%$ (indicated by a green dashed line). The error estimate exceeds the tolerance for a short moment after $t = 3\,s$. Such an overestimation can cause the execution of unnecessary full-order solves. Nonetheless, an overestimation is of less harm to the approximation quality since the exact error still meets the tolerance. 
\begin{figure}[H]
    \centering
    \includegraphics[width=0.65\textwidth]{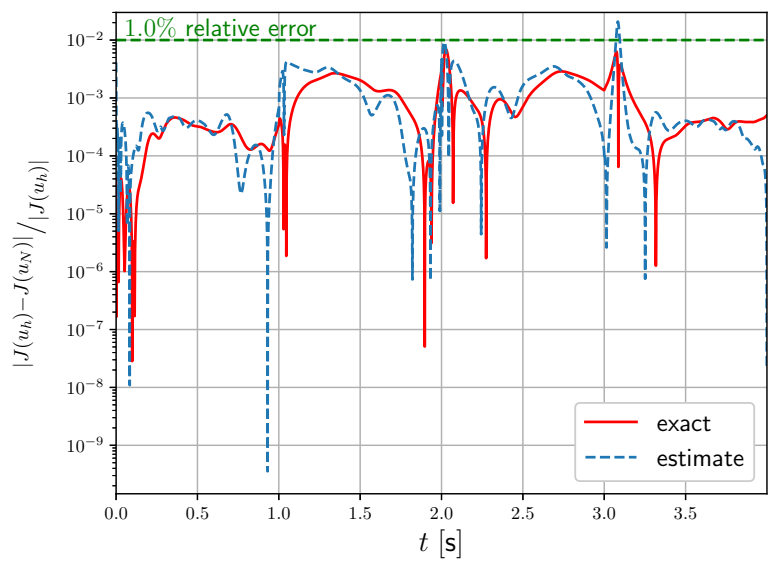}
    \caption{Temporal evolution of the time interval-wise error estimator compared to the true error for the 1+1D heat equation.}
    \label{fig:error_estimator_1d_heat}
\end{figure}
Table \ref{tab:comparison:1d_heat} gives an overview of simulation results for different error tolerances comprised between $0.1\%$ and $10\%$. The listed characteristics are: the relative error, computational speedup, the total  number of FOM solves, POD basis sizes for the primal and dual problem, prediction capability of error estimator, and the effectivity index from (\ref{eq:effectivity_index}). Here, the number of FOM solves sums up all primal and dual solves and the basis sizes are shown in the pattern primal/dual. The prediction capability is visualized by means of a confusion matrix. The prediction on each slab is assigned to one of the four cases: 
\begin{align*}
    &\text{error} > \text{tol} \land \text{estimate} < \text{tol} \; \quad | \quad  \;
    \text{error} < \text{tol} \land \text{estimate} > \text{tol} \; \quad | \quad \\
    &\text{error} > \text{tol} \land \text{estimate} > \text{tol} \; \quad | \quad \;
    \text{error} < \text{tol} \land \text{estimate} < \text{tol}.
\end{align*}

We note that the four possible scenarios are sorted according to the severity of the consequences of their occurrence. So, the first two cases indicate mispredictions of the estimator.
Here, the first case is the least desirable since then the error estimator underestimates the true error, which can lead to an insufficiently small reduced basis. The second case is less fatal since then the true error is being overestimated by the error estimator, which can cause the reduced basis to be slightly larger than necessary. The last two cases are less harmful since the estimate correctly predicts the error. However, the third case is also not optimal, since it shows that after the incremental basis enrichment, in the validation loop, there are still slabs on which the error tolerance is being exceeded. Therefore, we expect that for an efficient method (almost) all slabs fall in the last category, where the error tolerance is being met and the error estimate is also below the tolerance.

We observe that with a rise in the tolerance the relative error as well as the speedup increase. Note that the relative error is almost a magnitude smaller than the tolerance, which aligns with the results of Figure \ref{fig:error_estimator_1d_heat}. The difference in magnitude can be explained by the fact that the tolerance has to be met slabwise while the relative error is evaluated over the whole time domain. The speedup is explained by the decreasing amount of FOM solves and smaller POD bases for both the primal and dual problem w.r.t. the given tolerance. Furthermore, the estimator predicts the relationship of the error to the tolerance in approximately $98-99\%$ of the cases right with most of the incorrect predictions being overestimations. Similarly, for the effectivity index, a slight worsening can be seen with rising tolerance since then replacing the full-order dual solution in the error estimator with the reduced-order dual solution introduces additional errors.
\begin{table}[H]
    \centering
    \resizebox{\columnwidth}{!}{%
        \begin{tabular}{ |p{2.5cm}||p{2.5cm}|p{2.0cm}|p{2.5cm}|p{2.0cm}|p{2.8cm}|p{2.5cm}|  }
             \hline
             Tolerance & Relative error & Speedup & FOM solves & Basis size & Prediction & Effectivity \\
             \hline
              $0.1 \%$ & $0.0130 \%$  & $10.9$  & 68 & 39 | 36 &\phantom{0}0 | 38 | \phantom{0}0 | 5082& $1.0065$ \\
             $1 \%$ & $0.1210 \%$  & $	12.2$ & 40 & 25 | 22&\phantom{0}0 | 31 | \phantom{0}0 | 5089& $1.0071$ \\
              $2 \%$ & $0.3370 \%$  & $13.2$  & 38& 24 | 21	 &\phantom{0}0 | 41 | \phantom{0}0 | 5079& $1.0074$ \\
             $5 \%$ & $1.2019 \%$ & $15.2$ & 32& 21 | 18	 &28 | 48 | 18 | 5026& $1.0125$ \\
             $10 \%$ & $1.7645 \%$ & $18.6$ & 30 & 20 | 17&\phantom{0}0 | 73 | \phantom{0}0 | 5047& $1.0404$ \\
             \hline
        \end{tabular}
    }
    \caption{Performance of MORe DWR for the 1+1D heat equation depending on the tolerance in the goal functional.}
    \label{tab:comparison:1d_heat}
\end{table}
Finally, we demonstrate the incremental nature of our MORe DWR approach in Figure \ref{fig:basis_evolution_1d_heat}. In this context, we illustrate the on-the-fly basis generation by plotting the primal and dual reduced basis size over the time domain and compare its evolution for the tolerances of $1\%$ and $10\%$. The results indicate a steeper and more granular increase of both the primal and dual basis size if the tolerance is smaller. Nevertheless, we observe a steady basis size for all bases and tolerances after around 2 seconds. If we take the movement of the heat source into account, this is exactly the time the source needs to travel once through the spatial domain. Thus, after this, no new information is added to the system that would trigger a further basis enrichment.
\begin{figure}[H]
    \centering
    \includegraphics[width=0.485\textwidth]{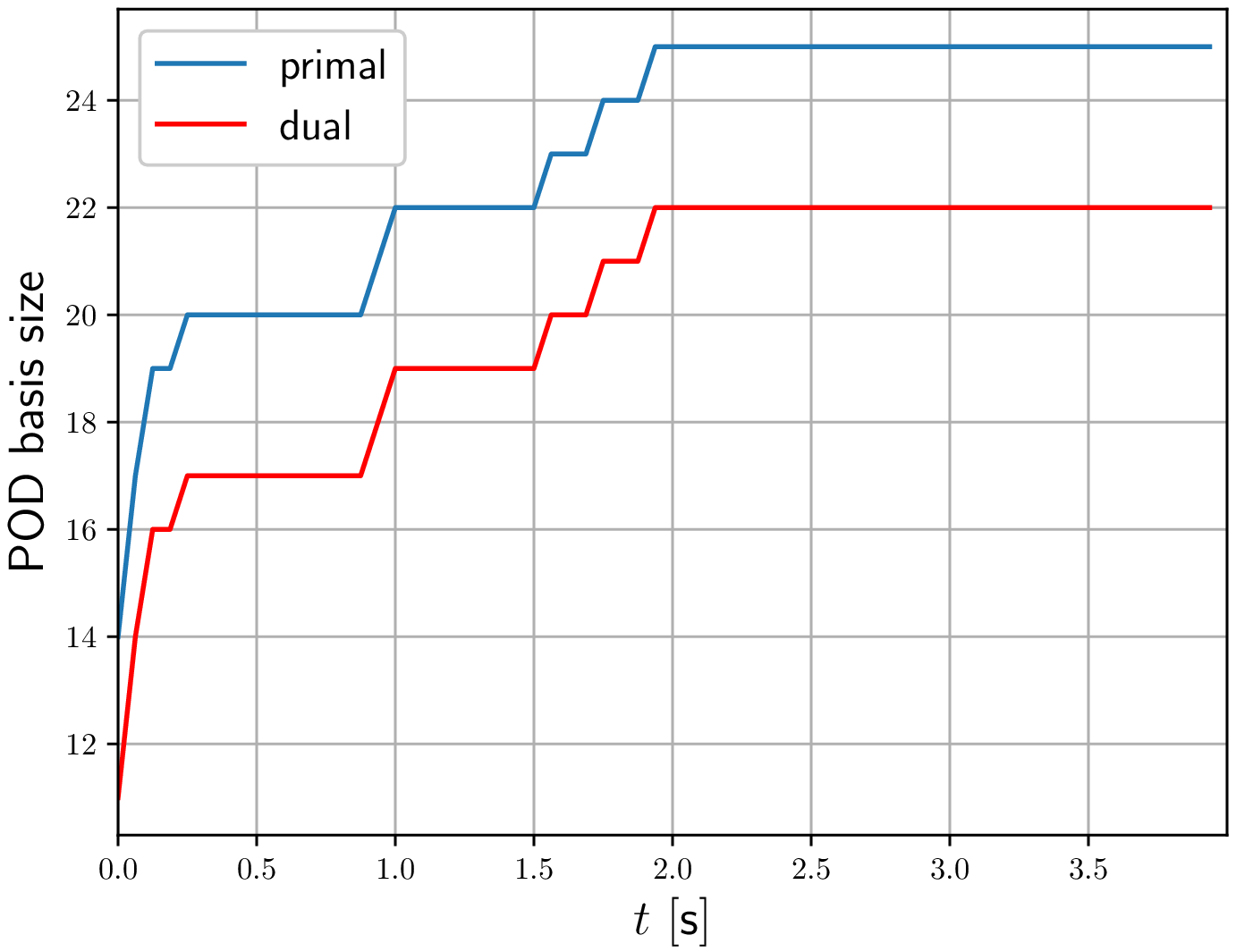}
    \quad
    \includegraphics[width=0.485\textwidth]{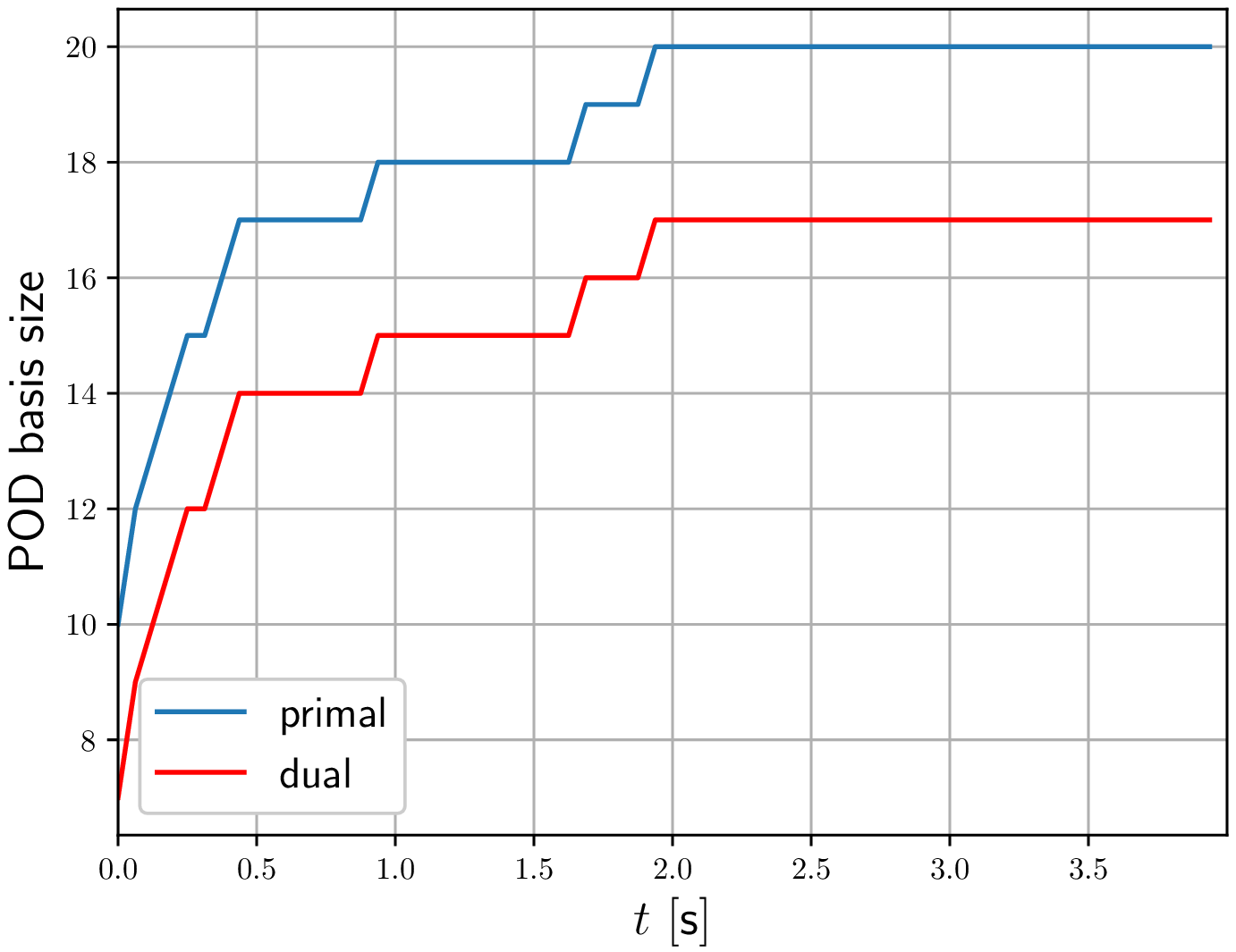}
    \caption{Temporal evolution of the reduced basis size for a relative error tolerance of $1\%$ (left) and $10\%$ (right) for the 1+1D heat equation.}
    \label{fig:basis_evolution_1d_heat}
\end{figure}
\subsection{2+1D Heat equation}
\label{sec:2d_heat}
In the second numerical experiment, we test MORe DWR on a 2+1D heat equation problem. We consider the spatial domain $\Omega = (0,1)^2$ and the temporal domain $I = (0,10)$. We create a moving heat source of oscillating temperature that rotates around the midpoint of the spatial domain $\Omega$ as shown in Figure \ref{fig:fom_2d_heat_snapshot}. For this, we use the right-hand side function
\begin{align*}
    f(t, x) := \begin{cases}
        \sin(4 \pi t)  & \text{if } (x_1 - p_1)^2 + (x_2 - p_2)^2 < r^2,\\
        0 & \text{else},
    \end{cases}
\end{align*}
with $x = (x_1, x_2)$, midpoint $p = (p_1, p_2) = (\frac{1}{2}+\frac{1}{4} \cos(2 \pi t), \frac{1}{2}+\frac{1}{4} \sin(2 \pi t))$ and radius of the trajectory $r=0.125$. In addition, a zero initial condition and homogeneous Dirichlet boundary conditions are applied. In contrast to the goal functional in Section \ref{subsec:1+1d_heat_numerical}, we test the method for a nonlinear cost functional  $J(u) := \frac{1}{10}\int_0^{10} \int_\Omega u(t,x)^2\ \mathrm{d}x\ \mathrm{d}t$.

\begin{figure}[H]
    \centering
     \subfloat[t = 0.074]{
    \includegraphics[width=4cm]{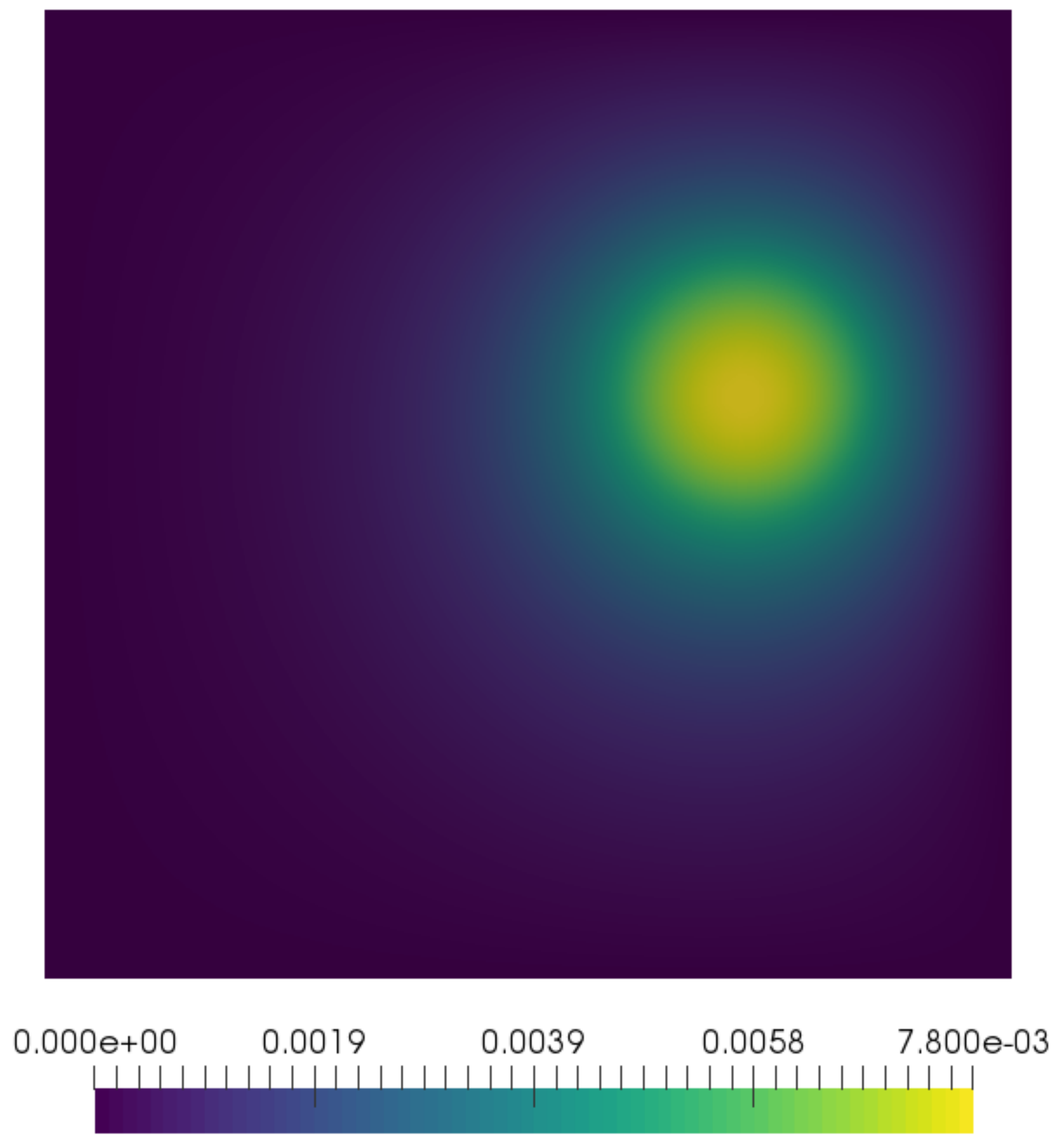}}%
     \subfloat[t = 0.416]{
    \includegraphics[width=4cm]{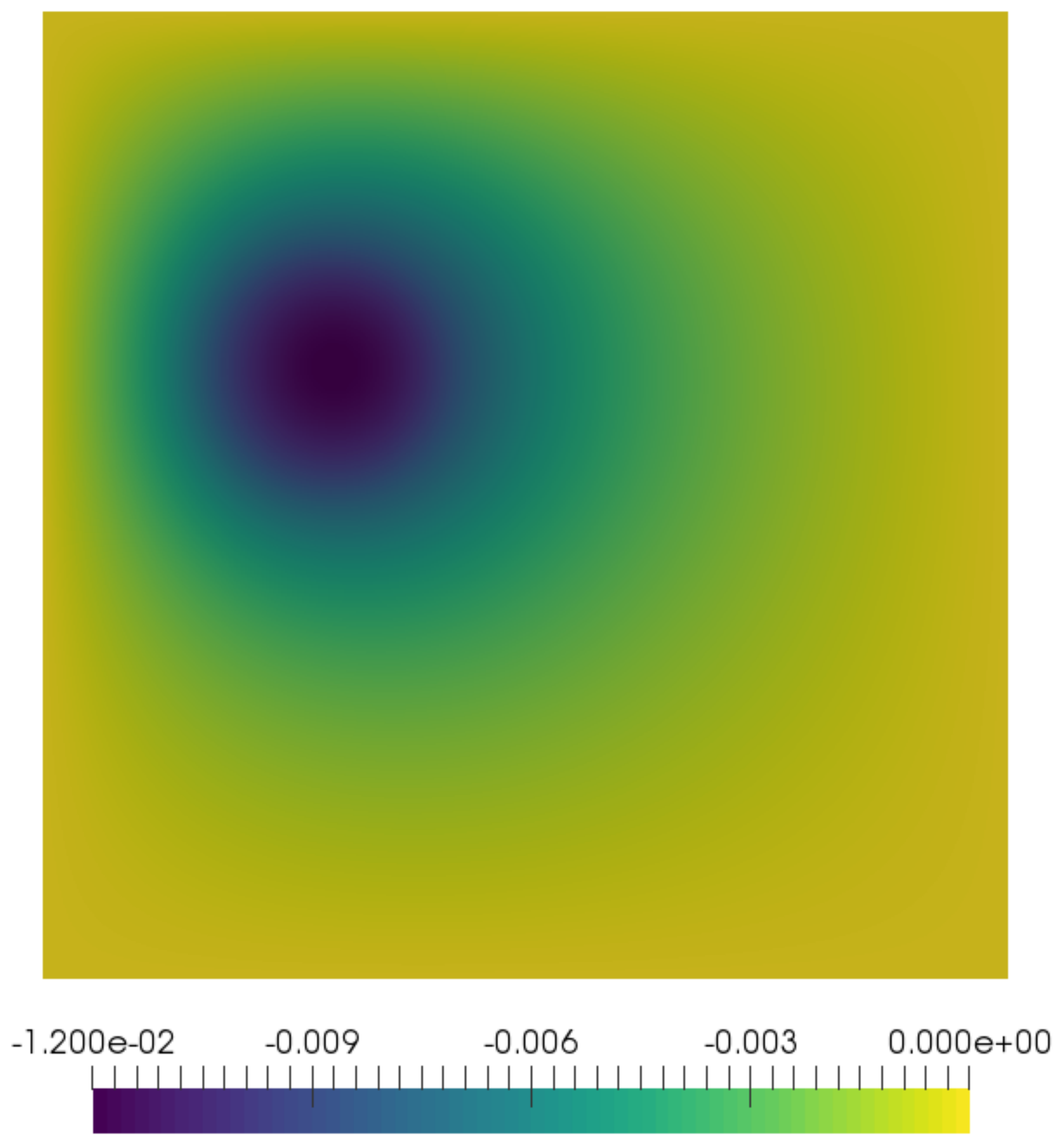}}%
    \\
    \subfloat[t = 0.611]{
    \includegraphics[width=4cm]{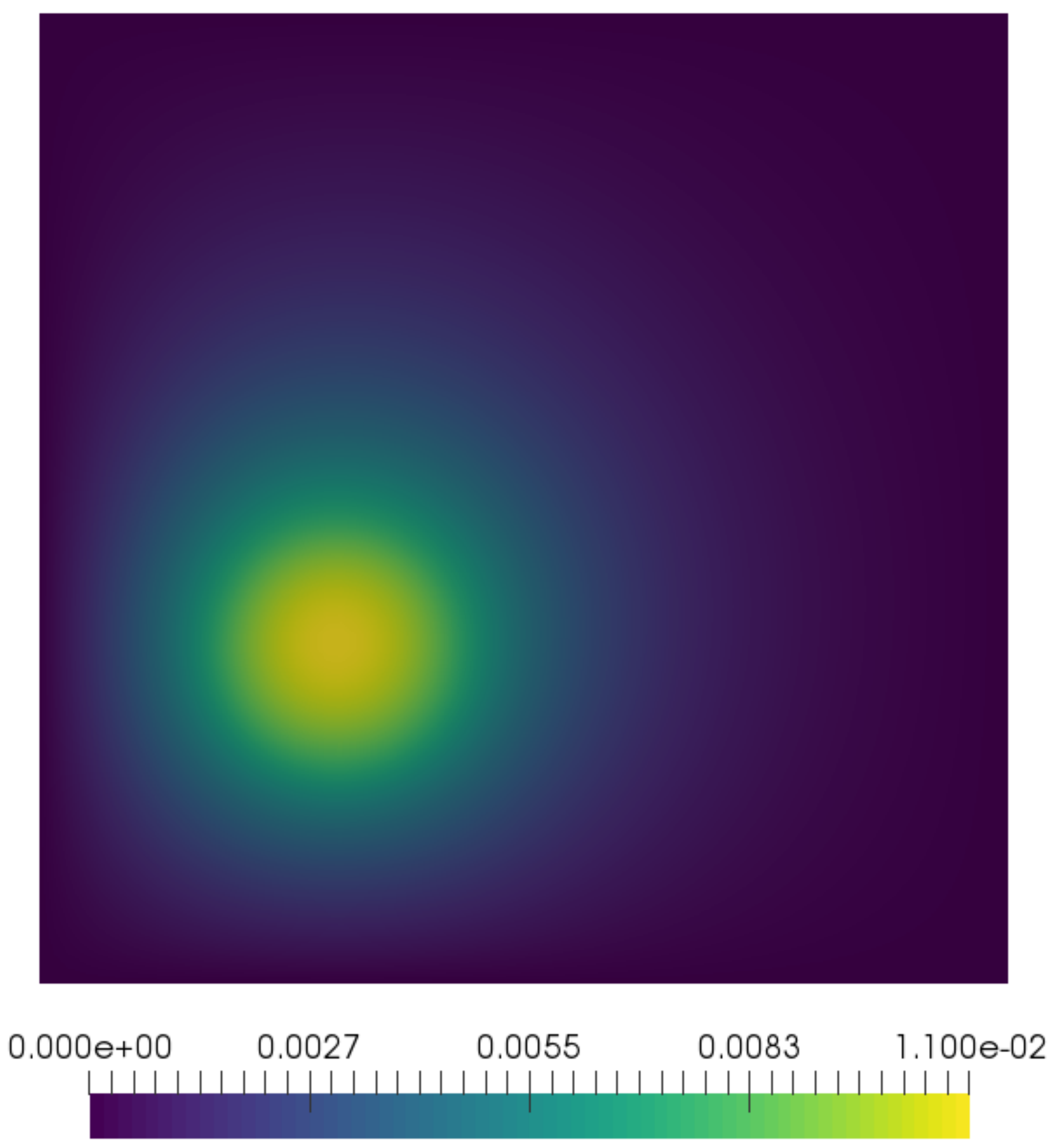}}%
    \subfloat[t = 0.885]{
    \includegraphics[width=4cm]{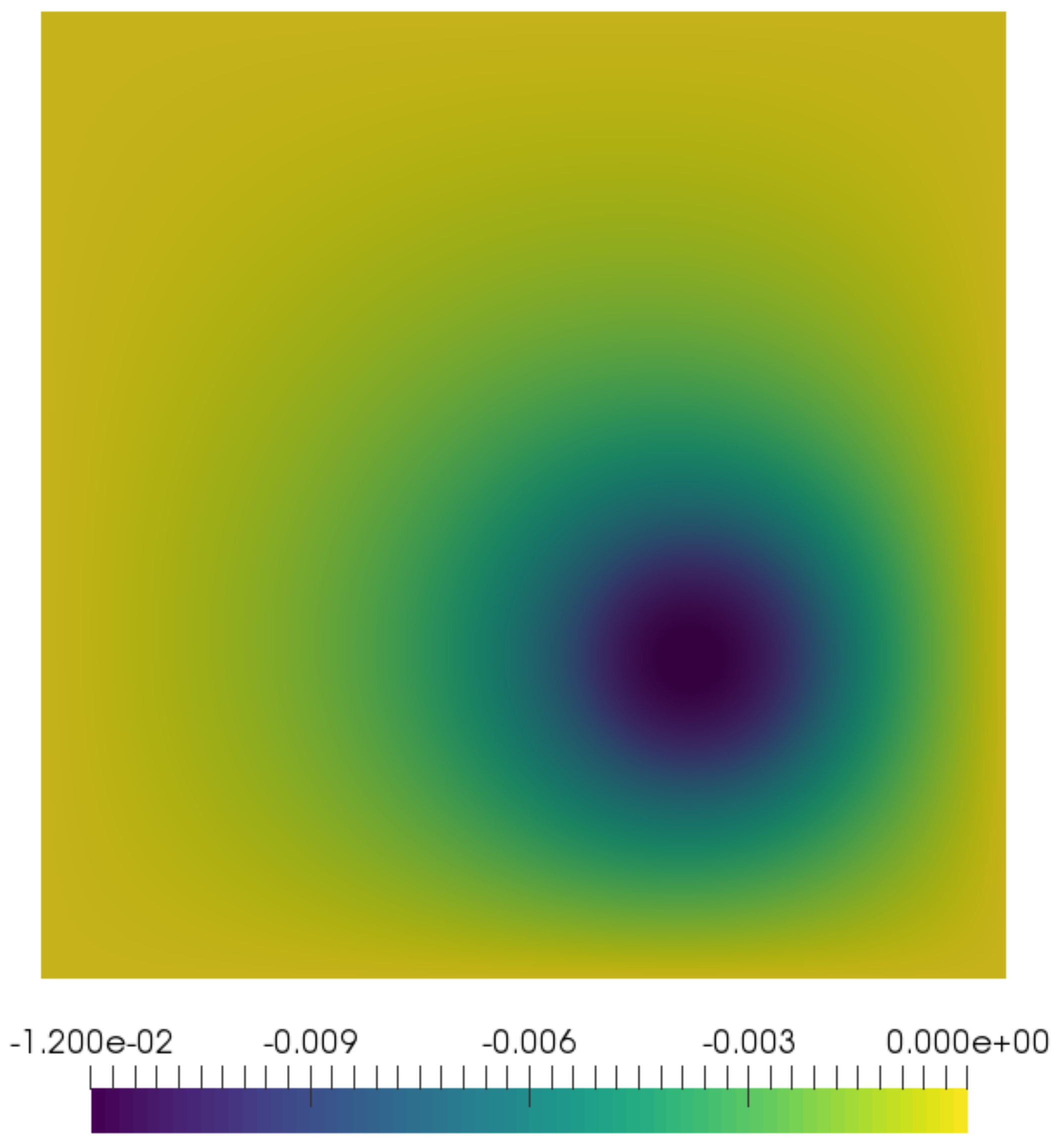}}%
    \caption{Full-order solution snapshots for the 2+1D heat equation.}
    \label{fig:fom_2d_heat_snapshot}
\end{figure}

For the reduced-order model, we choose that the primal and dual reduced bases have to preserve $\varepsilon = 1-10^{-8}$ of the information. Similar to the previous one-dimensional scenario, we resort to the relative error estimate $\eta^{{rel}}_{N,{S_{P_k}^l}}$ and allow errors up to a tolerance of $1\%$. The full-order model is characterized by $n = 4,225$ and 
$q = 4,096$ DoFs in space and time, respectively. This gives us a total of 
$n \cdot q = 17,305,600$ space-time degrees of freedom. Further, the temporal domain is split up into $M = 2,048$ time slabs. For the incremental ROM, we choose a total amount of $K = 128$ parent-slabs on which the slabs are evenly distributed, i.e. $L = 16$.

Firstly, in Figure \ref{fig:cost_functional_2d_heat} we compare the time trajectories of the goal functional restricted to each time slab for the full-order space-time solution $u_h$ and the reduced-order space-time solution $u_N$. 
It illustrates that both trajectories are not distinguishable from each other although the solution behavior is constantly changing. 
Furthermore, good approximation quality can also be observed when regarding the time-averaged cost functional. We obtain $J(u_h) = 6.4578 \cdot 10^{-5}$ and $J(u_N) = 6.4577\cdot 10^{-5}$ yielding a relative error of $\eta_{max} = 0.0016\%$.
This implies that the incremental ROM can replicate nonlinear cost functionals within a given tolerance. 
\begin{figure}[H]
    \centering
    \includegraphics[width=0.65\textwidth]{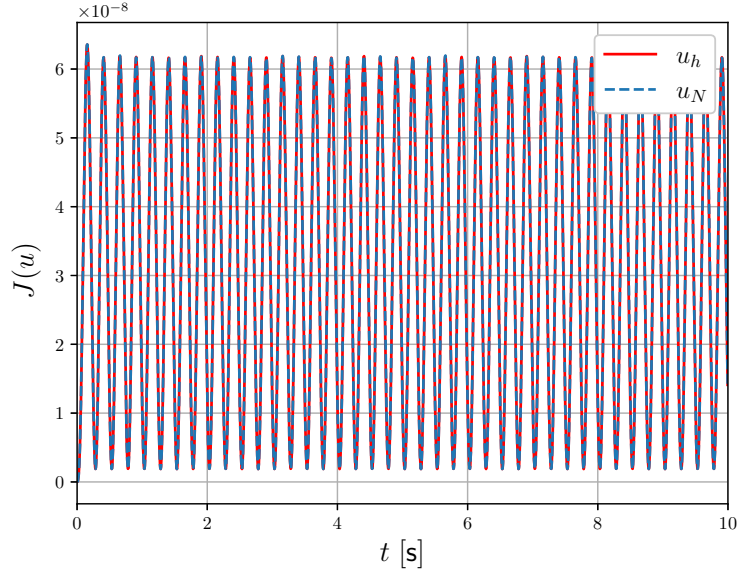}
    \caption{Temporal evolution of cost functional for 2+1D heat equation.}
    \label{fig:cost_functional_2d_heat}
\end{figure}
In Figure \ref{fig:error_estimator_2d_heat}, the exact temporal errors and their estimation on each slab are compared. Further, for illustration we indicate the error tolerance of $1\%$ in this plot. The results show that both the exact and estimated errors meet the given error tolerance on all slabs. Overall, the estimate shows a similar trajectory to the exact error. However, we can observe spikes in the exact error that are not completely covered by the estimation. Nevertheless, these deflections remain without consequences.
\begin{figure}[H]
    \centering
    \includegraphics[width=0.65\textwidth]{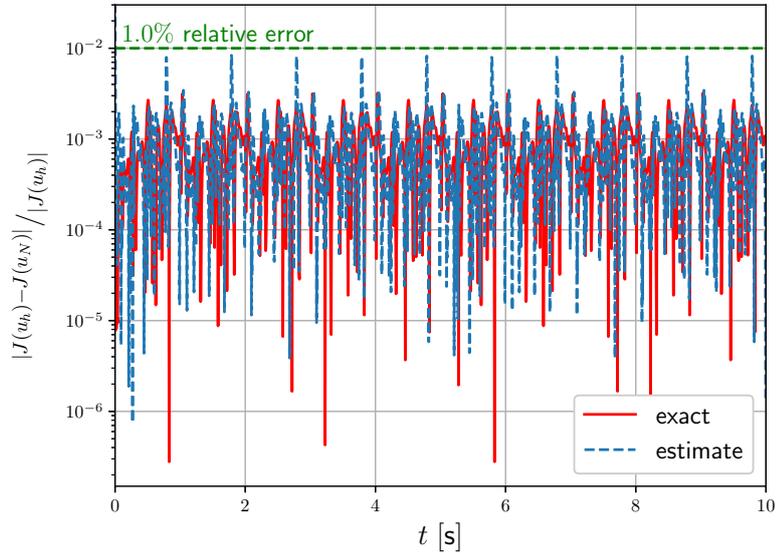}
    \caption{Temporal evolution of the time interval-wise relative error estimator compared to the true error for the 2+1D heat equation.}
    \label{fig:error_estimator_2d_heat}
\end{figure}
Table \ref{tab:comparison:2d_heat} presents simulation results for a range of error tolerances. The quantities we consider are the following: the relative error, computational speedup, the total  number of FOM solves, POD basis sizes for the primal and dual problem, prediction capability of the error estimator, and the effectivity index. For definitions of these quantities, we refer to Section \ref{subsec:1+1d_heat_numerical}. We can observe that with a rise in the tolerance the relative error as well as the speedup increase. Again, the relative error is much smaller than the tolerance. Note in contrast to the 1D linear scenario the relaxation of the error tolerance has a greater impact on the speedup. 
This can be explained by the evolution of the amount of FOM solves and the POD bases w.r.t. the given tolerance. Furthermore, the estimator predicts the relationship of the error to the tolerance in approximately $94-99\%$ of the cases right with most of the incorrect predictions being overestimations. An exception exists for $\text{tol} = 10\%$ where a drop of $5\%$ in the prediction capability can be observed indicating the dual basis is too small to accurately estimate the error. An adapted tolerance for the information content of the dual basis could counteract that problem. Nevertheless, the obtained reduced cost functional still meets the error tolerance. The largest difference to the linear case holds the evaluation of the effectivity index. We observe that in contrast to the linear case, the effectivity indices show larger fluctuations around $1$, which have been expected due to the  nonlinear cost functional. However, the effectivity indices are still in an acceptable range yielding good results. 
Finally, we observe that for a large tolerance of $10 \%$ we have a few mispredictions, i.e. on $79$ slabs the true error is greater than the tolerance while the estimated error is smaller than the tolerance, and on $28$ slabs the error estimator is greater than the tolerance while the true error is smaller than the tolerance. Additionally, for this tolerance we have $17$ slabs on which both true and estimated errors are larger than the tolerance. This decay of MORe DWR performance can be explained by the replacement of the fine dual solution $z^{\text{fine}}$ in the DWR error estimator by the coarse dual solution $z^{\text{coarse}}$. If we make the POD bases for the primal and dual problems too small, then this approximation might cause additional errors and lead to a worse performance of our method.
\begin{table}[H]
    \centering
    \resizebox{\columnwidth}{!}{%
        \begin{tabular}{ |p{2.5cm}||p{2.5cm}|p{2.0cm}|p{2.5cm}|p{2.0cm}|p{2.8cm}|p{2.5cm}|  }
             \hline
             Tolerance & Relative error & Speedup & FOM solves & Basis size & Prediction & Effectivity \\
             \hline
              $0.1\%$ & $0.0019 \%$ & $7.7$ &150& 92 | 78 &\phantom{0}0 | 35 | \phantom{0}0 | 2013& $0.7524$ \\
             $1\%$ & $0.0017 \%$ & $27.5$ &80 & 55 | 44 &\phantom{0}0 | \phantom{0}1 | \phantom{0}0 | 2047& $0.2771$ \\
              $2\%$ & $0.0628\%$  & $29.6$ & 66 & 47 | 36 &\phantom{0}0 | \phantom{0}9 | \phantom{0}0 | 2039& $3.9181$ \\
             $5\%$ & $0.9162 \%$ & $44.8$  & 44& 33 | 25 &\phantom{0}0 | \phantom{0}1 | \phantom{0}0 | 2047& $1.2254$ \\
             $10\%$ & $0.9243 \%$ & $50.0$ &38& 31 | 23 &79 | 28 | 17 | 1924& $1.5474$ \\
             \hline
        \end{tabular}
    }   
    \caption{Incremental reduced-order modeling summary for the 2+1D heat equation depending on the tolerance in the goal functional.}
    \label{tab:comparison:2d_heat}
\end{table}
Lastly, Figure \ref{fig:basis_evolution_2d_heat} sketches the incremental nature of the MORe DWR approach. The on-the-fly basis generation is shown by plotting the primal and dual reduced basis size over the time domain and comparing its evolution for the tolerances of $1\%$ and $10\%$. The results indicate a steep increase of both the primal and dual basis sizes in the first second of the simulation that reflects one round trip of the oscillating heat source through the spatial domain. Again, for a more restrictive tolerance, the size of a reduced basis enhances in a faster fashion. After the first round trip of the heat source, the basis size remains almost unchanged with only one basis enlargement for the tolerance of $1\%$ at around $t = 4~\text{s}$. This is grounded in the periodic behavior of the chosen numerical experiment that does not add any further information to the system. Thus, less or no further basis enrichments have to be performed to meet the given error tolerance.

\begin{figure}[H]
    \centering
    \includegraphics[width=0.485\textwidth]{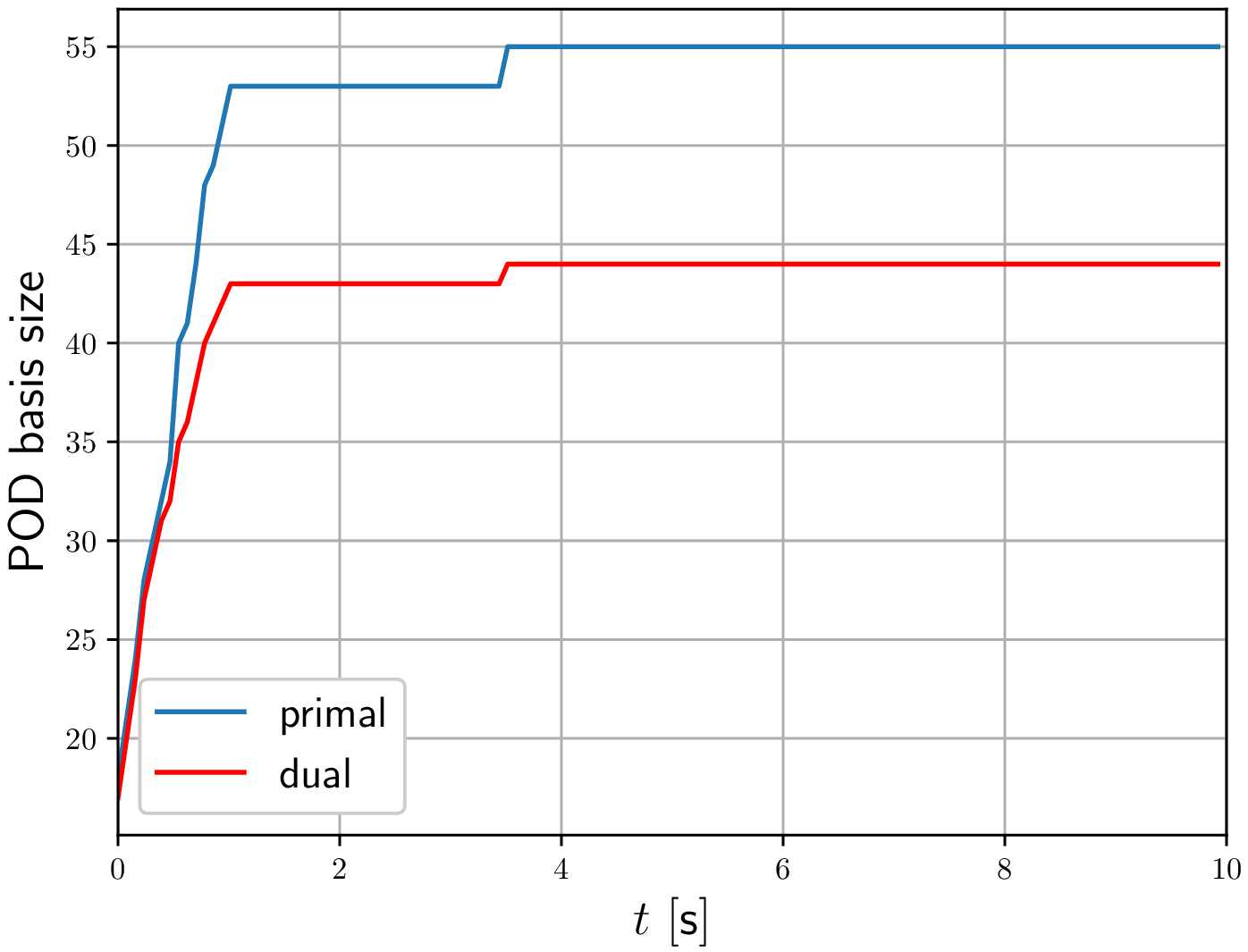}
    \quad
    \includegraphics[width=0.485\textwidth]{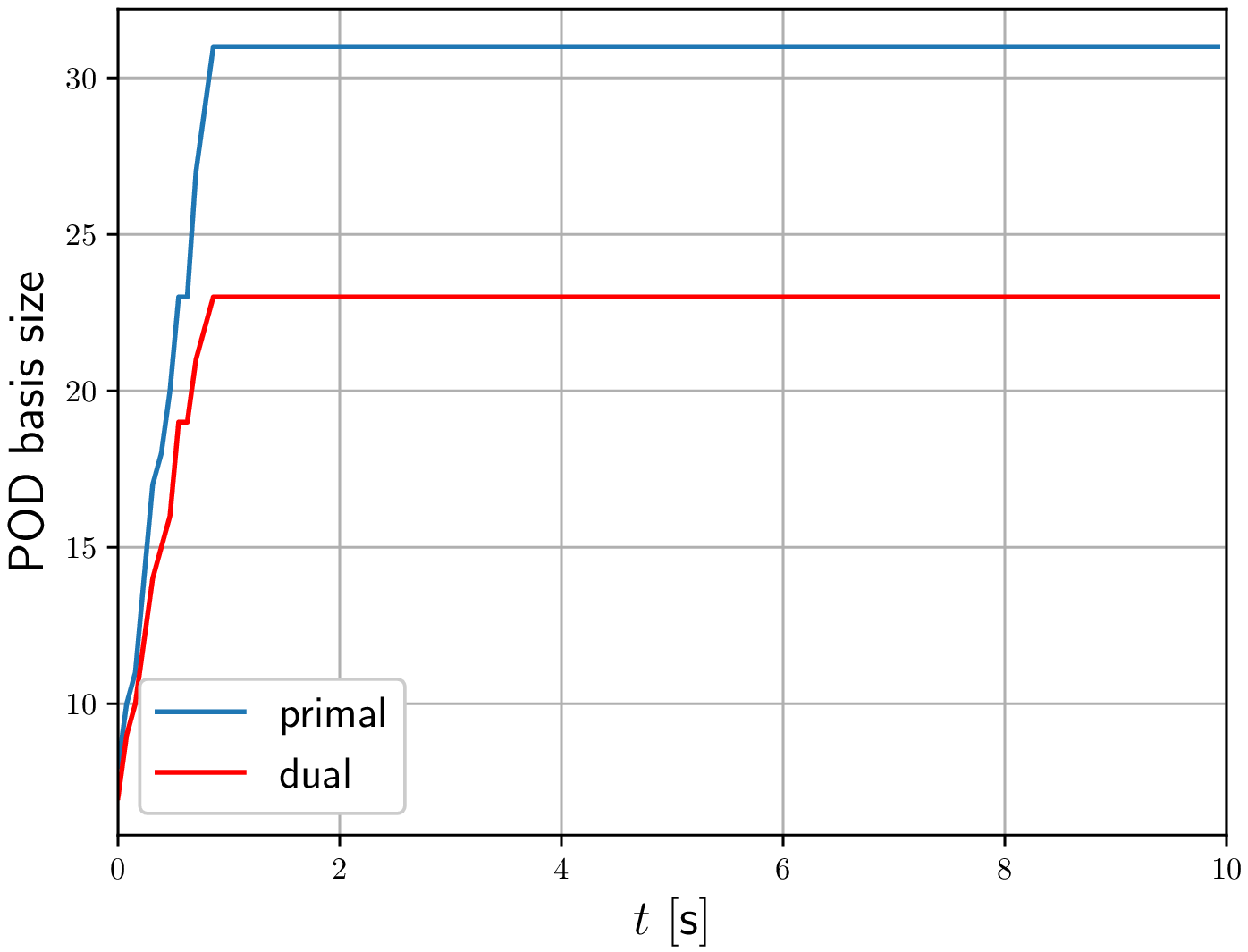}
    \caption{Temporal evolution of the reduced basis size for a relative error tolerance of $1\%$ (left) and $10\%$ (right) for the 2+1D heat equation.}
    \label{fig:basis_evolution_2d_heat}
\end{figure}
\subsection{3+1D Elastodynamics equation}
In the third numerical experiment, we choose Formulation \ref{form:variational_elasto} and investigate the method on a 3+1D elastodynamics problem. We consider a rectangular beam spanning the spatial domain $\Omega = (0,6) \times (0,1) \times (0,1)$. Further, the temporal domain $I = (0,40)$ is regarded. We induce an oscillation in the vertical direction by defining a force $f(t,x)$ acting on the upper boundary of the beam $\Gamma_\text{up} = (0,6) \times (0,1) \times \{x_3 = 1\}$. In the first part of the experiment the beam is lifted up by means of the acting force as shown in Figure \ref{fig:fom_snapshot_elastodynamics}. Thereafter, the force is slowly eliminated such that the beam begins to swing. 

For this, we use 
\begin{align*}
    g(t) := \begin{cases}
        f_\text{max} \frac{t}{t_1} & x_3 = 1 \; \land \; t \leq t_1,\\
        f_\text{max} \left(1-\frac{t-t_1}{t_2-t_1}\right) & x_3 = 1 \; \land \; t_1 < t \leq t_2,\\
        0 & \text{else},
    \end{cases}
\end{align*}
with maximal acting force $f_\text{max} = 0.5$ and $t_1 = 5$ and $t_2 = 6$ being the time points until the force increases or decreases, respectively. Together with the beam being clamped at the boundary $\Gamma_{\text{clamped}} = \{ x_1 = 0\}\times (0,1) \times (0,1)$ this yields the boundary conditions in Section \ref{sec:elasto_equation}
\begin{align*}
    u &= 0  \qquad \text{in } I \times \Gamma_{\text{clamped}}, \\
    v &= 0  \qquad \text{in } I \times \Gamma_{\text{clamped}}, \\
    \sigma(u) \cdot n &= 0  \qquad \text{in } I \times \partial \Omega \setminus (\Gamma_{\text{clamped}} \cup \Gamma_\text{up}), \\
    \sigma(u) \cdot n &= g(t) \qquad \text{in } I \times \Gamma_{up}.
\end{align*}
Furthermore,  the homogeneous initial conditions are given by 
\begin{align*}
    u(0) &= 0 \qquad \text{in } \Omega, \\
    v(0) &= 0 \qquad \text{in } \Omega.
\end{align*}
We choose the time-averaged stress acting on the clamped boundary $\Gamma_\text{clamped}$ denoted by $J(u) := \frac{1}{40} \int_0^{40} \int_\Omega \sigma(u(t,x)) \cdot n\ \mathrm{d}x\ \mathrm{d}t$ as the cost functional. 
For the reduced-order model, we decide that the primal and dual reduced bases have to preserve $\varepsilon = 1-10^{-11}$  of the information. Again, we resort to the relative error estimate $\eta^{{rel}}_{N,{S_{P_k}^l}}$ and allow errors up to a tolerance of $1\%$. The full-order model is characterized by $n = 702$ and $q = 4,800$ DoFs in space and time, respectively. This gives us a total of $n \cdot q = 561,600$ space-time degrees of freedom. Further, the temporal domain is split up to $M = 1,600$ time slabs. For the incremental ROM, we choose a total amount of $K = 80$ parent-slabs on which the slabs are evenly distributed, i.e. $L = 20$.
\begin{figure}[H]
    \centering
    \includegraphics[width=8cm]{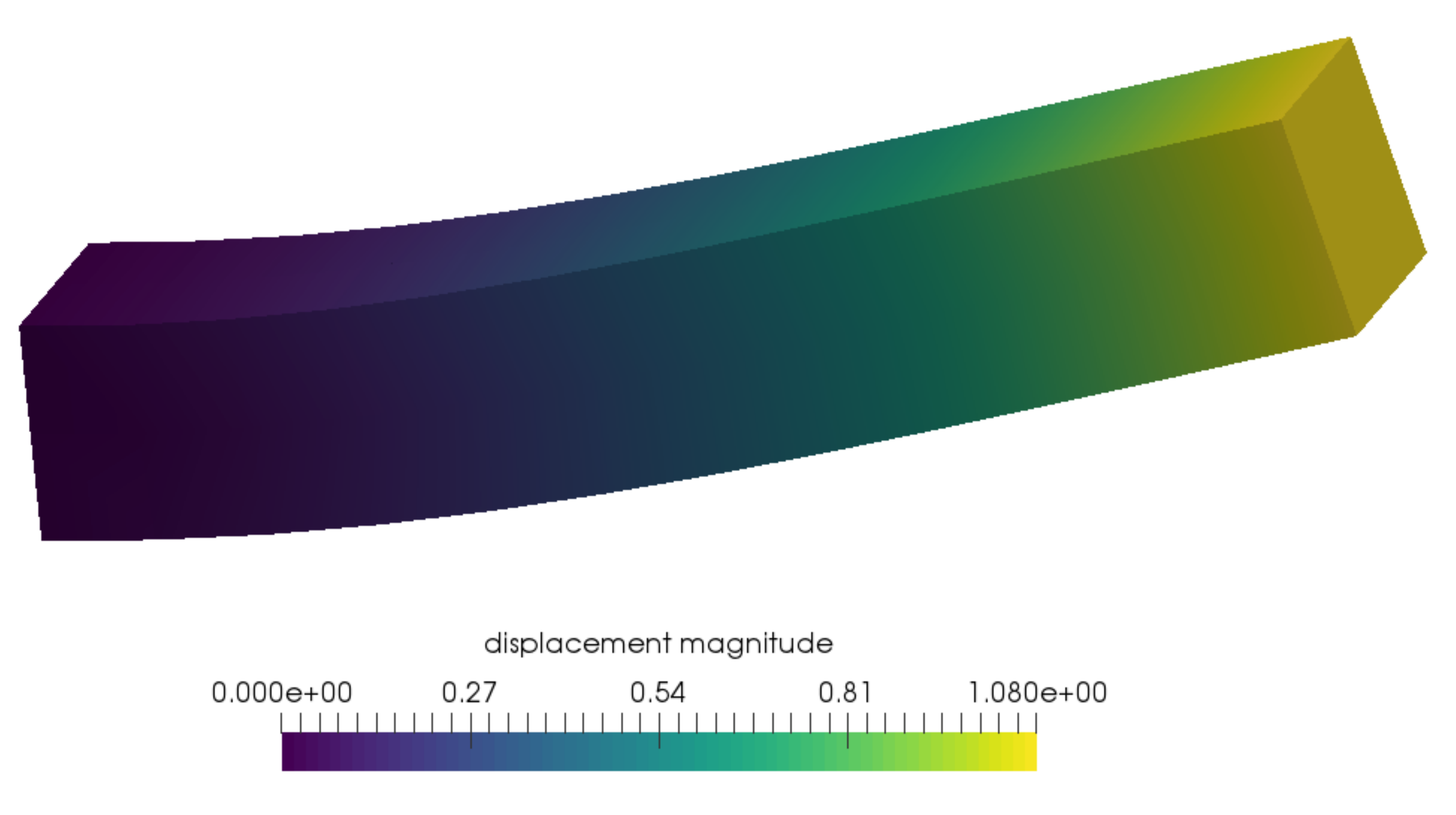}
    \caption{Full-order solution snapshot at t = 5.75 for the elastodynamics equation.}
    \label{fig:fom_snapshot_elastodynamics}
\end{figure}

We compare the time trajectories of the goal functional restricted to each time slab for the full-order space-time solution $u_h$ and the reduced-order space-time solution $u_N$ in Figure \ref{fig:cost_functional_elastodynamics}. The results show that both trajectories are indistinguishable from each other and the oscillating behavior can be mimicked by the reduced cost functional.
Furthermore, the good approximation quality can also be observed when regarding the time-averaged cost functional. We obtain $J(u_h) = -3.1114$ and $J(u_N) = -3.1115$ yielding a relative error of $\eta_{max} = 0.0035\%$ which is smaller than the desired error tolerance of $1\%$.
\begin{figure}[H]
    \centering
    \includegraphics[width=0.65\textwidth]{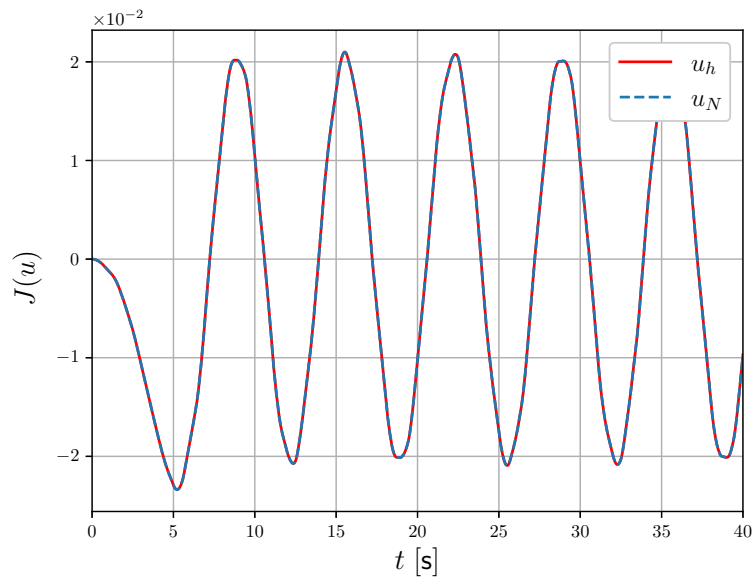}
    \caption{Temporal evolution of goal functional of the 3+1D elastodynamics equation.}
    \label{fig:cost_functional_elastodynamics}
\end{figure}
In Figure \ref{fig:error_estimator_elastodynamics}, we plot the exact temporal errors and their estimation on each slab for comparison. For illustration purposes, we indicate the error tolerance of $1\%$ in this plot. The results show that both quantities are on average in the same order of magnitude while the standard deviation of the real error appears to be larger. Thus, there exist spikes in the error that are not captured by the error estimation. Most of the time, this has no consequence but on one slab the error tolerance is exceeded slightly. 
\begin{figure}[H]
    \centering
    \includegraphics[width=0.65\textwidth]{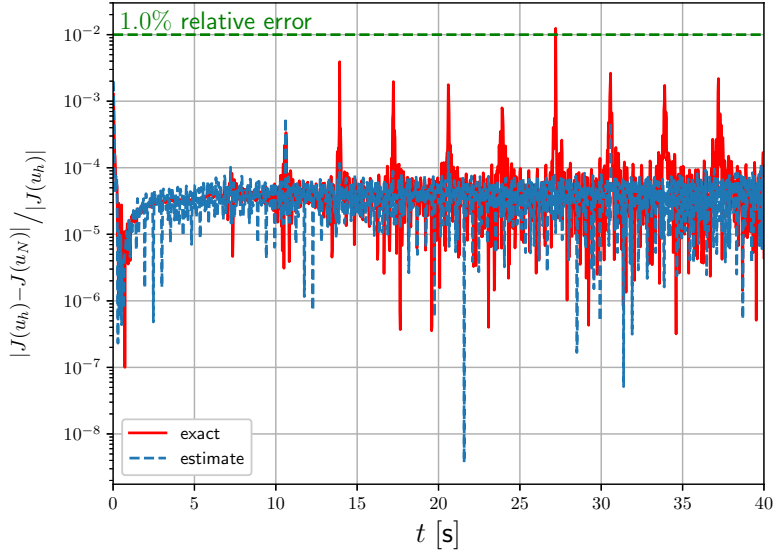}
    \caption{Temporal evolution of the time interval-wise relative error estimator compared to the true error of the 3+1D elastodynamics equation.}
    \label{fig:error_estimator_elastodynamics}
\end{figure}
In order to investigate the violation of the error tolerance, we present in Table \ref{tab:comparison:elastodynamics} simulation results for a range of error tolerances. Therefore, we show the relative error, the prediction measures and the effectivity indices to examine the error estimation. Further, the computational speedup, total number of FOM solves and the POD basis sizes for the primal and dual problems are displayed. For definitions of these quantities, we refer to Section \ref{subsec:1+1d_heat_numerical}. We can observe that while most of the mispredictions are poor underestimations of the error, there are only a few of them. In addition, the effectivity indices are near to the optimum of $1$ and the relative errors meet the tolerance in all scenarios. However, a small decay in the effectivity indices can be recognized for the reduced-order models with larger tolerance.
Additionally, when reviewing the performance measurement, we can determine differences to the previous heat problems. The resulting speedups as well as the FOM solves are near constant for all tolerances. Only for a tolerance of $10 \%$ we do see further improvements in speedups and basis size reduction. We also observe that the total amount of FOM solves and the size of the POD bases are not monotonically decreasing w.r.t. the error tolerance like in the heat equation setting. A reason for this behavior can be assigned to the behavior of the error itself.
Using a smaller tolerance can lead to more reduced basis enrichment early on. Larger tolerances lead to smaller initial reduced basis and so errors further on due to the small basis size, which is compensated by enlarging the basis in a later stage.

\begin{table}[H]
    \centering
\resizebox{\columnwidth}{!}{%
    \begin{tabular}{ |p{2.5cm}||p{2.5cm}|p{2.0cm}|p{2.5cm}|p{2.0cm}|p{2.8cm}|p{2.5cm}|  }
             \hline
             Tolerance & Relative error & Speedup & FOM solves & Basis size & Prediction & Effectivity \\
             \hline
             $0.1 \%$ & $0.0042 \%$ & $11.2$ & 44 &\phantom{0}98 | 79&26 | \phantom{0}0 | \phantom{0}1 | 1573& $1.0029$ \\
              $1 \%$ & $0.0035 \% $ & $12.1$ & 46 &107 | 82 &\phantom{0}1 | \phantom{0}0 | \phantom{0}0 | 1599&$1.0001$ \\
             $2 \%$ & $0.0001 \%$ & $10.6$ & 46 &113 | 82 &\phantom{0}0 | \phantom{0}0 | \phantom{0}0 | 1600& $0.9953$ \\
             $5 \%$ & $0.0040 \%$ & $12.6$ & 48&101 | 86&\phantom{0}0 | \phantom{0}0 | \phantom{0}0 | 1600&  $	0.9999$ \\
             $10 \%$ & $0.0200 \%$ &  $14.9$ & 38&\phantom{0}89 | 84&44 | \phantom{0}1 | \phantom{0}0 | 1555&$	0.9801$ \\
             \hline
        \end{tabular}
    }
    \caption{Incremental reduced-order modeling summary for the 3+1D elastodynamics equation depending on the tolerance in the goal functional.}
    \label{tab:comparison:elastodynamics}
\end{table}
Finally, the incremental nature of our MORe DWR approach is depicted in Figure \ref{fig:basis_evolution_3d_elasto}. We allow insights into the on-the-fly basis generation by plotting the primal and dual reduced basis size over the time domain and comparing its evolution for the tolerances of $1\%$ and $10\%$. Similar to the previous scenarios, we observe a steep increase in both the primal and dual basis sizes at the beginning of the simulation. In addition, we see further changes in the reduced basis sizes in the second half of the simulation. We see again that a tighter tolerance yields a larger primal reduced basis, i.e. more refinements of the reduced basis have been performed. Another curiosity of this numerical experiment is that for a tolerance of $1\%$, the reduced dual basis shrinks slightly at $t \approx 10~\text{s}$, due to an iPOD update with a snapshot that carries a lot of information.

\begin{figure}[H]
    \centering
    \includegraphics[width=0.485\textwidth]{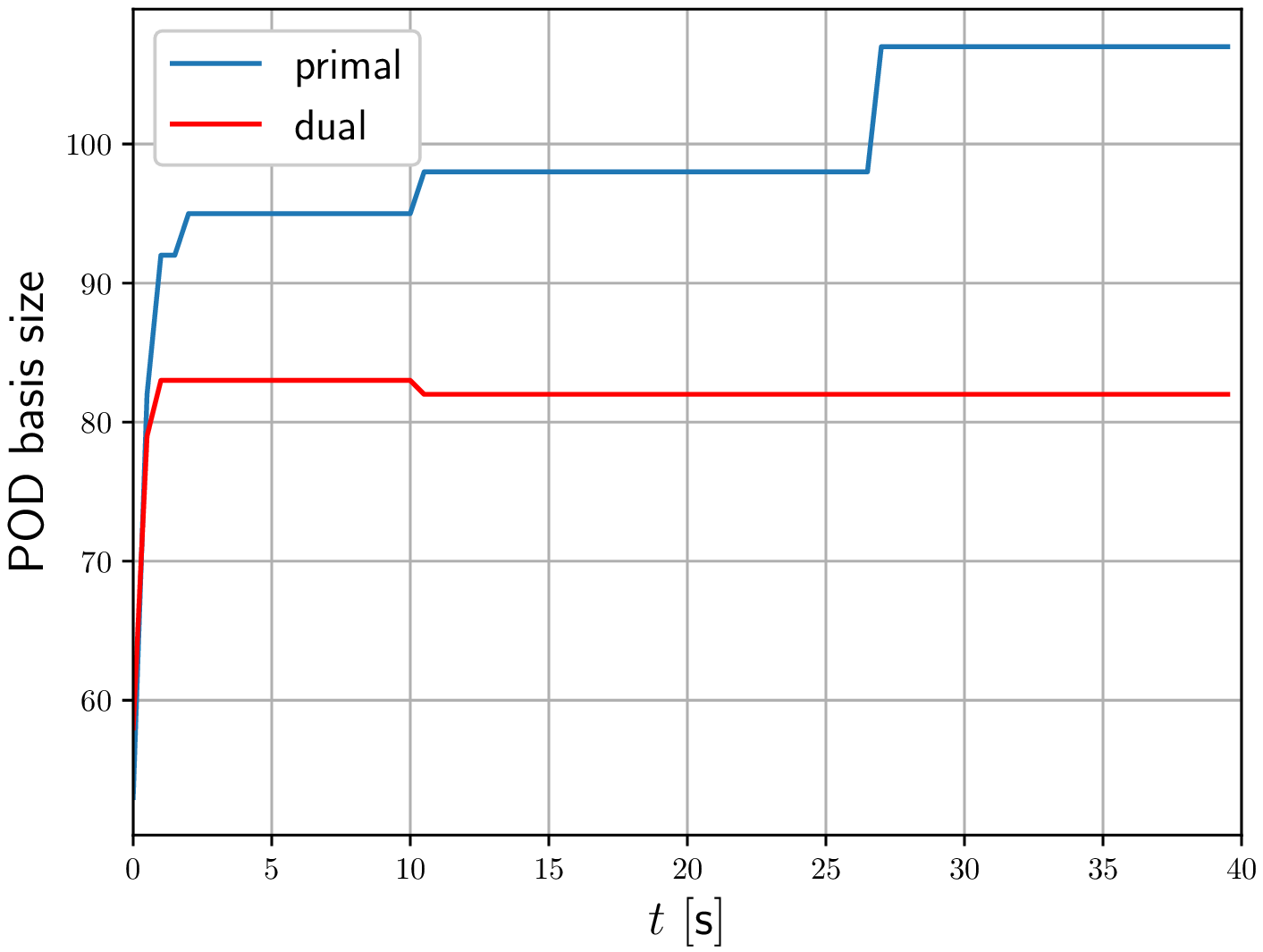}
    \quad
    \includegraphics[width=0.485\textwidth]{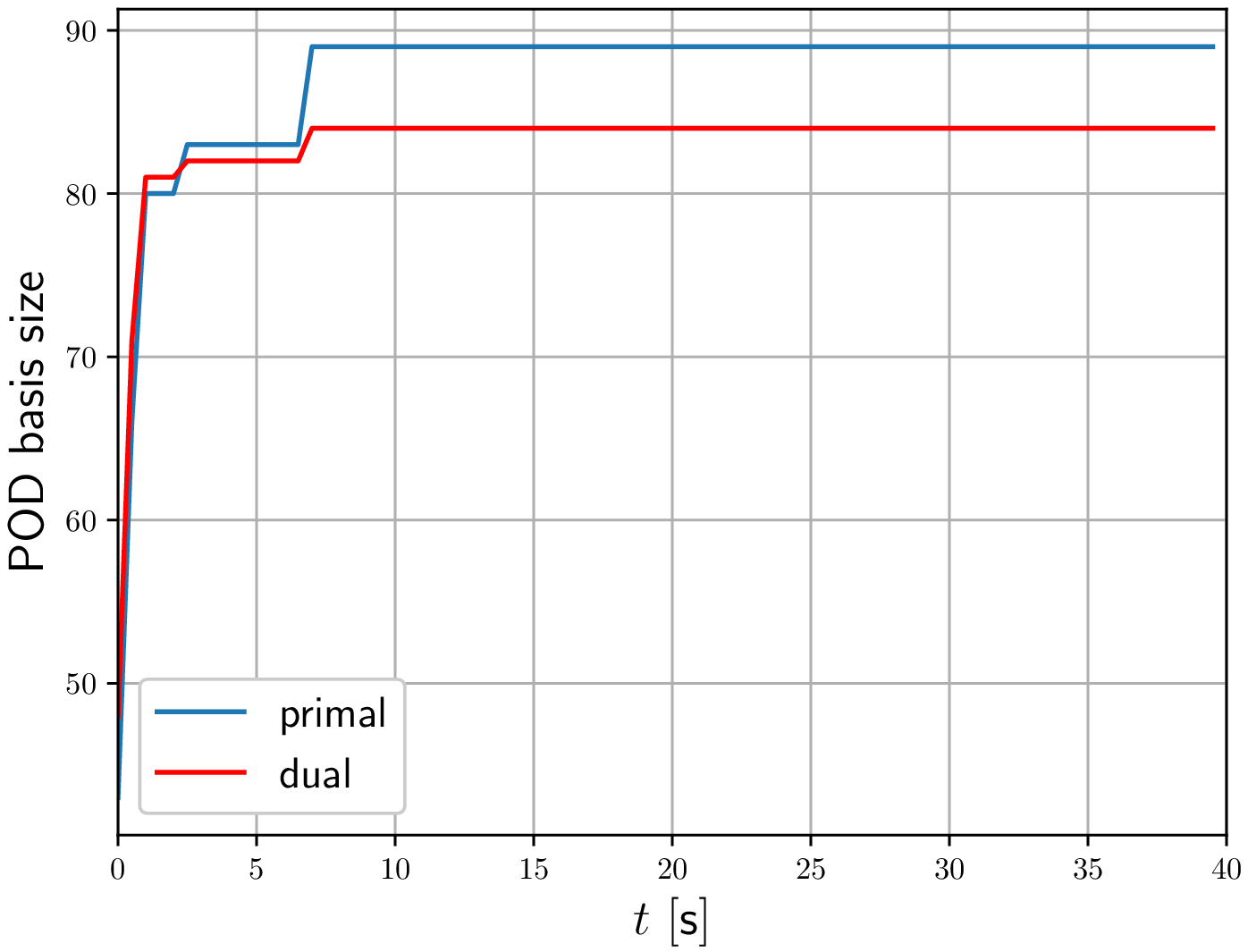}
    \caption{Temporal evolution of the reduced basis size for a relative error tolerance of $1\%$ (left) and $10\%$ (right) for the 3+1D elastodynamics equation.}
    \label{fig:basis_evolution_3d_elasto}
\end{figure}
\section{Conclusion and outlook}
\label{sec:conclusion}
In this work, we proposed a novel incremental POD-ROM method with on-the-fly basis enrichment based on space-time DWR error estimates for linear PDEs, namely the heat equation and elastodynamics, and linear and nonlinear goal functionals. This methodology can be applied to a wide class of problems since its efficiency has been demonstrated in Section \ref{sec:numerical_tests}.
The effectivity indices for linear problems are almost exactly one, which makes the error estimates reliable in practice, and for nonlinear goal functionals, we had also good effectivity indices. For nonlinear PDEs and goal functionals possibly full DWR is needed. Additionally, we had speedups of up to 50, while the error between the FOM and the ROM solution remained within our prescribed tolerance. Consequently, e.g. the expensive high-fidelity computations in the offline stage of the reduced basis method could be replaced by our incremental POD method. An interesting aspect for future work would be the extension of this method to dynamical, adaptive spatial meshes to further speed up the computations. 

\section*{Acknowledgements}
The authors acknowledge the funding of the German Research Foundation (DFG) within the framework of the International Research Training Group on  Computational Mechanics Techniques in High Dimensions GRK 2657 under Grant Number 433082294. In addition, we thank Hendrik Geisler (Leibniz University Hannover, GRK 2657) for fruitful discussions and comments. The support of the French-German University through the French-German Doctoral college "Sophisticated Numerical and Testing Approaches" (CDFA-DFDK 19-04) is also acknowledged.

\appendix

\section{Space-time linear system and dG(r) time-stepping formulation for elastodynamics}
\label{sec:elasto_time_stepping}

The space-time discretization of the elastodynamics equation on a slab with a single temporal element and a $\dG(r)$ in time discretization is discussed in this appendix. Using the fully discrete variational formulation \ref{form:elasto_with_jumps} of elastodynamics, we arrive at the linear equation system
\begin{align}\label{eq:elasto_dG_linear_system}
    \left[ C_k \otimes M_h + M_k \otimes K_h + D_k^1 \otimes M_h \right] U_m = F_m + \left[D_k^2 \otimes M_h\right] U_{m-1}    
\end{align}
where we use the spatial matrices
\begin{align*}
    M_h &= \left\{ (\varphi_h^{v,(j)},\varphi_h^{u,(i)}) + (\varphi_h^{u,(j)},\varphi_h^{v,(i)}) \right\}_{i,j = 1}^{\# \text{DoFs}(\mathcal{T}_h)}, \\
    K_h &= \left\{ (\sigma(\varphi_h^{u,(j)}),\nabla_x\varphi_h^{u,(i)}) + (\varphi_h^{v,(j)},\varphi_h^{v,(i)}) \right\}_{i,j = 1}^{\# \text{DoFs}(\mathcal{T}_h)}
\end{align*}
and the temporal matrices
\begin{align*}
    M_k &= \left\{ \int_{I_m}\varphi_k^{(j)}\cdot \varphi_k^{(i)}\ \mathrm{d}t \right\}_{i,j = 1}^{\# \text{DoFs}(I_m)}, \\
    C_k &= \left\{ \int_{I_m}\partial_t\varphi_k^{(j)}\cdot \varphi_k^{(i)}\ \mathrm{d}t \right\}_{i,j = 1}^{\# \text{DoFs}(I_m)}, \\
    D_k^1 &= \begin{pmatrix}
        1 & 0 & \cdots & 0 \\
        0 & 0 & & \\
        \vdots & & \ddots & \\
        0 & & & 0
    \end{pmatrix}, \qquad
    D_k^2 = \begin{pmatrix}
        0 & \cdots & 0 & 1\\
         & & 0 & 0 \\
        & \iddots  & & \vdots \\
        0 & & & 0
    \end{pmatrix}.
\end{align*}
Here, the solution vector $U_m$ for the $\dG(r)$ method in time with temporal quadrature points $t_1, \dots, t_{r+1}$ is given by
\begin{align*}
    U_{m} = \begin{pmatrix}
        U_{m} (t_1) \\
        \vdots \\
        U_{m} (t_{r+1})
    \end{pmatrix} = \begin{pmatrix}
        u_{m} (t_1) \\
        v_{m} (t_1) \\
        \vdots \\
        u_{m} (t_{r+1}) \\
        v_{m} (t_{r+1})
    \end{pmatrix}.
\end{align*}
To derive the $\dG(r)$ time-stepping formulation, we now only need to evaluate the temporal matrices $M_k$ and $C_k$ by integrating over $(0,k)$, where $k := t_m - t_{m-1}$ is the time step size, and by plugging in the $\dG-Q^r$ basis functions on $(0,k)$, which coincide with the $Q^r$ basis functions since we only have one single element and use Gauss-Lobatto quadrature.

\subsection{dG(1) formulation of elastodynamics}
\label{ref:subsec_app}
By inserting the basis functions $\varphi_k^{(1)} = 1 - \frac{t}{k}, \varphi_k^{(2)} = \frac{t}{k}$ into the temporal matrices $M_k$ and $C_k$ we get
\begin{align*}
    \int_0^k\varphi_k^{(1)}\cdot \varphi_k^{(1)}\ \mathrm{d}t  &= \int_0^k \left( 1 - \frac{t}{k} \right)^2 \ \mathrm{d}t = \int_0^k 1 - \frac{2t}{k} + \frac{t^2}{k^2}\ \mathrm{d}t = \frac{k}{3} = \int_0^k\varphi_k^{(2)}\cdot \varphi_k^{(2)}\ \mathrm{d}t, \\
    \int_0^k\varphi_k^{(1)}\cdot \varphi_k^{(2)}\ \mathrm{d}t &= \int_0^k\varphi_k^{(2)}\cdot \varphi_k^{(1)}\ \mathrm{d}t = \int_0^k \left( 1 - \frac{t}{k} \right) \cdot \frac{t}{k} \ \mathrm{d}t = \int_0^k \frac{t}{k} - \frac{t^2}{k^2}\ \mathrm{d}t = \frac{k}{6},
\end{align*}
as well as
\begin{align*}
    \int_0^k \partial_t\varphi_k^{(1)}\cdot \varphi_k^{(2)}\ \mathrm{d}t  &= \int_0^k \partial_t \left( 1 - \frac{t}{k} \right) \cdot \frac{t}{k} \ \mathrm{d}t = \int_0^k - \frac{t}{k^2} \ \mathrm{d}t = -\frac{1}{2} = \int_0^k\partial_t \varphi_k^{(1)}\cdot \varphi_k^{(1)}\ \mathrm{d}t, \\
    \int_0^k\partial_t \varphi_k^{(2)}\cdot \varphi_k^{(2)}\ \mathrm{d}t &= \int_0^k \partial_t\left(\frac{t}{k}\right) \cdot \frac{t}{k} \ \mathrm{d}t = \int_0^k \frac{t}{k^2}\ \mathrm{d}t = \frac{1}{2} = \int_0^k\partial_t \varphi_k^{(2)}\cdot \varphi_k^{(1)}\ \mathrm{d}t.
\end{align*}
Consequently, the $\dG(1)$ time-stepping formulation for elastodynamics reads
\begin{align*}
    \left[ \frac{1}{2}\begin{pmatrix}
        1 & 1 \\ -1 & 1
    \end{pmatrix} \otimes M_h + 
    \frac{k}{6}\begin{pmatrix}
        2 & 1 \\
        1 & 2
    \end{pmatrix} \otimes K_h \right] \begin{pmatrix}
        U_m(t_{m-1}) \\ U_m(t_{m})
    \end{pmatrix} = \begin{pmatrix}
        F_m(t_{m-1}) + U_{m-1}(t_{m-1})M_h \\ F_m(t_{m})
    \end{pmatrix},
\end{align*}
where we use the fact that the temporal quadrature points for $\dG(1)$ are $t_{m-1}$ and $t_m$.

\subsection{dG(2) formulation of elastodynamics}
Repeating the procedure from Section \ref{ref:subsec_app} with quadratic basis functions, we arrive at the $\dG(2)$ time-stepping formulation for elastodynamics
\begin{align*}
    \left[ \frac{1}{6}\begin{pmatrix}
        3 & 4 & -1 \\ -4 & 0 & 4 \\ 1 & -4 & 3
    \end{pmatrix} \otimes M_h + 
    \frac{k}{30}\begin{pmatrix}
        4 & 2 & -1 \\
        2 & 16 & 2 \\
        -1 & 2 & 4
    \end{pmatrix} \otimes K_h \right] \begin{pmatrix}
        U_m(t_{m-1}) \\ U_m(t_{m-\frac{1}{2}}) \\ U_m(t_{m})
    \end{pmatrix} = \begin{pmatrix}
        F_m(t_{m-1}) + U_{m-1}(t_{m-1})M_h \\ F_m(t_{m-\frac{1}{2}}) \\ F_m(t_{m})
    \end{pmatrix},
\end{align*}
where we use the fact that the temporal quadrature points for $\dG(2)$ are $t_{m-1}$, $t_{m-\frac{1}{2}} := t_{m-1} + \frac{k}{2}$ and $t_m$.
\begin{remark}
    The $\dG(1)$and $\dG(2)$ formulations can also be found in Section 7.1 and Section 7.2 in \cite{Richter2013efficient} for an ODE model. 
\end{remark}

%

\bibliographystyle{abbrv}
\bibliography{lit.bib}

\end{document}